\documentclass[11pt]{article}
\usepackage[utf8]{inputenc}
\usepackage{Radiality_Preserving}

\usepackage{dsfont}
    \usepackage{authblk}

\allowdisplaybreaks[4]

\newcommand{\RNum}[1]{\uppercase\expandafter{\romannumeral #1\relax}}

\newcommand{\Addresses}{{% additional braces for segregating \footnotesize
  \bigskip\bigskip \small

  Bjoern Bringmann, \textsc{University of California, Los Angeles, Department of Mathematics, 520 Portola Plaza, Los Angeles, CA 90095}\\\nopagebreak
  Email address: \texttt{bringmann@math.ucla.edu}

}}
\title{Almost sure scattering for the radial energy critical nonlinear wave equation in three dimensions}
\author{Bjoern Bringmann}
\begin{document}
\pagenumbering{arabic}

\maketitle
\let\thefootnote\relax
\let\thefootnote\relax
\begin{abstract} \noindent
We study the Cauchy problem for the radial energy critical nonlinear wave equation in three dimensions. Our main result proves almost sure scattering for radial initial data below the energy space. In order to preserve the spherical symmetry of the initial data, we construct a radial randomization that is based on annular Fourier multipliers. We then use a refined radial Strichartz estimate to prove probabilistic Strichartz estimates for the random linear evolution. The main new ingredient in the analysis of the nonlinear evolution is an interaction flux estimate between the linear and nonlinear components of the solution. 
We then control the energy of the nonlinear component by a triple bootstrap argument involving the energy, the Morawetz term, and the interaction flux estimate. 
\end{abstract}

\tableofcontents

\section{Introduction}
We consider the defocusing nonlinear wave equation (NLW) in three dimensions
\begin{equation}\label{in:eq_nlw}
 \begin{cases}
-\partial_{tt} u+ \Delta u =u^5 ~, \qquad \qquad \quad (t,x)\in \mathbb{R}\times \rthree\\
u(0,x)= f(x) \in \dot{H}_x^s(\rthree), \qquad \partial_t u(0,x) = g(x) \in \dot{H}_x^{s-1}(\rthree). 
\end{cases}
\end{equation}
The flow of nonlinear wave equation \eqref{in:eq_nlw} conserves the energy 
\begin{equation}
E[u](t) := \int_{\rthree} \frac{|\nabla u(t,x)|^2}{2} + \frac{|\partial_t u(t,x)|^2}{2} + \frac{u(t,x)^6}{2} \dx ~.
\end{equation}
Since the scaling-symmetry \( u(t,x) \mapsto u_\lambda(t,x) = \lambda^{-\tfrac{1}{2}} u(t/\lambda,x/\lambda) \) of \eqref{in:eq_nlw} leaves the energy invariant, we call \eqref{in:eq_nlw} energy critical. Using Sobolev embedding, it follows that the energy of the initial data is finite if and only if \( (f,g) \in \dot{H}_x^1(\rthree) \times L_x^2(\rthree) \). Therefore, we refer to \( \dot{H}_x^1(\rthree) \times L_x^2(\rthree) \) as the energy space. \\
If the initial data has finite energy, the nonlinear wave equation \eqref{in:eq_nlw} is now well-understood. In a series of seminal papers by several authors \cite{BG99,Grillakis90,Grillakis92,Rauch81,SS93,SS94,Strauss68,Struwe88,Tao06b}, it was proven that solutions to \eqref{in:eq_nlw} exist globally, obey global spacetime bounds, and scatter as \( t \mapsto \pm \infty \). In contrast, the equation is ill-posed if the initial data only lies in \( H_x^s(\rthree) \times H_x^{s-1}(\rthree) \) for some \( 0 < s < 1\). For instance, it has been shown in  \cite{CCT03} that solutions to \eqref{in:eq_nlw} exhibit norm-inflation with respect to the \( H_x^s \times H_x^{s-1}\)-norm. 
Consequently, this shows that we cannot construct local solutions of \eqref{in:eq_nlw} with initial data in \(  H_x^s \times H_x^{s-1}\) by a contraction mapping argument. \\
In recent years, there has been much interest in determining whether bad behaviour such as norm inflation is generic or only occurs for exceptional initial data. To answer this questions, multiple authors have studied solutions to dispersive equations with randomized initial data.  In the following discussion, we will focus on the Wiener randomization, and we refer the reader to the introduction of \cite{Pocovnicu17} as well as \cite{Bourgain94,Bourgain96,BT08I,BT08II,NORS12,TT10} for related works.  \\
Let us first recall the definition of the Wiener randomization from \cite{BOP2014,LM13}. We denote by \( Q=[-\frac{1}{2},\frac{1}{2})^d \) the unit cube centered at the origin. The family of translates \( \{ Q-k \}_{k\in \mathbb{Z}^d} \) forms a partition of \(\mathbb{R}^d \) (see. Fig \ref{fig:partition_rd}). By convolving the indicator function \( \chi_Q \) with a smooth and compactly supported kernel, we can construct a function \( \psi \in C^\infty_c(\rd) \) s.t. 
\begin{equation*}
\psi|_{[-\frac{1}{4},\frac{1}{4})^d}\equiv 1, \quad \psi|_{\rd \backslash [-1,1)^d} \equiv 0, \quad \text{and} \quad \sum_{k\in \mathbb{Z}^d} \psi(\xi-k) = 1 ~. 
\end{equation*}
Then, any function \( f \in L^2_x(\rd) \) can be decomposed in frequency space as 
\begin{equation*}
\widehat{f}(\xi) = \sum_{k\in \mathbb{Z}^d} \psi(\xi-k) \widehat{f}(\xi)~.
\end{equation*}
If \( \{ g_k \}_{k\in \mathbb{Z}^d} \) is a family of independent standard complex-valued Gaussians, the Wiener randomization \( f^\omega_W \) of \( f \) defined as
\begin{equation*}
\widehat{f^\omega_W}(\xi) := \sum_{k\in \mathbb{Z}^d} g_k(\omega) \psi(\xi-k) \widehat{f}(\xi)~. 
\end{equation*}
Thus, \( f^\omega_W \) is a random linear combination of functions whose Fourier transform is supported in unit-scale cubes. The Wiener randomization has been used to prove almost sure local and global well-posedness of nonlinear wave equations below the scaling-critical regularity. In \cite{LM13,LM16}, Lührmann and Mendelson proved the almost sure global well-posedness of energy subcritical nonlinear wave equations in \( \mathbb{R}^3 \). The first probabilistic result on the energy critical NLW was obtained by Pocovnicu in \cite{Pocovnicu17}, which treated the dimensions \( d=4,5 \). This method was extended by Oh and Pocovnicu \cite{OP16} to the three-dimensional case. In addition to nonlinear wave equations, the Wiener randomization has also been applied to nonlinear Schrödinger equations (NLS). 
Bényi, Oh, and Pocovnicu \cite{BOP2015,BOP2014,BOP17} proved the almost sure local well-posedness of the cubic NLS in \( \mathbb{R}^d \). This method was then extended by Brereton \cite{Brereton16} to the quintic NLS in \( \rd \). In \cite{BOP17b}, the authors proved the almost sure global well-posedness of the energy critical NLS in dimensions \( d=5,6 \).
 However, the global well-posedness results above do not give any information on the asymptotic behaviour of the solutions. \\
In contrast, Dodson, Lührmann, and Mendelson \cite{DLM17,DLM18} proved almost sure scattering for the energy critical NLW. Their result holds in dimension \( d=4 \) and requires that the original initial data (before the randomization) is spherically symmetric. The main idea is to control the energy-increment of the nonlinear component of \( u \) by a bootstrap argument involving both the energy and a Morawetz term.  The spherical symmetry is needed since the Morawetz estimate is centered around the origin. However, the Wiener randomization breaks the spherical symmetry, so that \( f^\omega_W \) is no longer radial. This method was subsequently extended to the energy critical NLS in dimension \( d=4 \) by \cite{DLM18,KMV17}. \\
In this work, we introduce a radial randomization that preserves the spherical symmetry of the initial data. To this end, let us first define a family of annular Fourier multipliers. 

\begin{figure}[t!]
\begin{center}
\begin{tikzpicture}[scale=0.8]
%Wiener
%Boxes blue
\draw[thick, fill opacity= 0.5, fill=blue] (-1/2,-1/2) -- (1/2,-1/2) -- (1/2,1/2) -- (-1/2,1/2) -- (-1/2,-1/2);
\draw[xshift=2cm, yshift=0cm, thick, fill opacity= 0.5, fill=blue] (-1/2,-1/2) -- (1/2,-1/2) -- (1/2,1/2) -- (-1/2,1/2) -- (-1/2,-1/2);
\draw[xshift=-2cm, yshift=0cm, thick, fill opacity= 0.5, fill=blue] (-1/2,-1/2) -- (1/2,-1/2) -- (1/2,1/2) -- (-1/2,1/2) -- (-1/2,-1/2);
\draw[xshift=-2cm, yshift=2cm, thick, fill opacity= 0.5, fill=blue] (-1/2,-1/2) -- (1/2,-1/2) -- (1/2,1/2) -- (-1/2,1/2) -- (-1/2,-1/2);
\draw[xshift=-2cm, yshift=-2cm, thick, fill opacity= 0.5, fill=blue] (-1/2,-1/2) -- (1/2,-1/2) -- (1/2,1/2) -- (-1/2,1/2) -- (-1/2,-1/2);
\draw[xshift=0cm, yshift=2cm, thick, fill opacity= 0.5, fill=blue] (-1/2,-1/2) -- (1/2,-1/2) -- (1/2,1/2) -- (-1/2,1/2) -- (-1/2,-1/2);
\draw[xshift=0cm, yshift=-2cm, thick, fill opacity= 0.5, fill=blue] (-1/2,-1/2) -- (1/2,-1/2) -- (1/2,1/2) -- (-1/2,1/2) -- (-1/2,-1/2);
\draw[xshift=2cm, yshift=-2cm, thick, fill opacity= 0.5, fill=blue] (-1/2,-1/2) -- (1/2,-1/2) -- (1/2,1/2) -- (-1/2,1/2) -- (-1/2,-1/2);
\draw[xshift=2cm, yshift=2cm, thick, fill opacity= 0.5, fill=blue] (-1/2,-1/2) -- (1/2,-1/2) -- (1/2,1/2) -- (-1/2,1/2) -- (-1/2,-1/2);
\draw[xshift=1cm, yshift=1cm, thick, fill opacity= 0.5, fill=blue] (-1/2,-1/2) -- (1/2,-1/2) -- (1/2,1/2) -- (-1/2,1/2) -- (-1/2,-1/2);
\draw[xshift=-1cm, yshift=1cm, thick, fill opacity= 0.5, fill=blue] (-1/2,-1/2) -- (1/2,-1/2) -- (1/2,1/2) -- (-1/2,1/2) -- (-1/2,-1/2);
\draw[xshift=1cm, yshift=-1cm, thick, fill opacity= 0.5, fill=blue] (-1/2,-1/2) -- (1/2,-1/2) -- (1/2,1/2) -- (-1/2,1/2) -- (-1/2,-1/2);
\draw[xshift=-1cm, yshift=-1cm, thick, fill opacity= 0.5, fill=blue] (-1/2,-1/2) -- (1/2,-1/2) -- (1/2,1/2) -- (-1/2,1/2) -- (-1/2,-1/2);
%Boxes red
\draw[xshift=1cm, yshift=0cm, thick, fill opacity= 0.5, fill=red] (-1/2,-1/2) -- (1/2,-1/2) -- (1/2,1/2) -- (-1/2,1/2) -- (-1/2,-1/2);
\draw[xshift=1cm, yshift=2cm, thick, fill opacity= 0.5, fill=red] (-1/2,-1/2) -- (1/2,-1/2) -- (1/2,1/2) -- (-1/2,1/2) -- (-1/2,-1/2);
\draw[xshift=1cm, yshift=-2cm, thick, fill opacity= 0.5, fill=red] (-1/2,-1/2) -- (1/2,-1/2) -- (1/2,1/2) -- (-1/2,1/2) -- (-1/2,-1/2);
\draw[xshift=0cm, yshift=1cm, thick, fill opacity= 0.5, fill=red] (-1/2,-1/2) -- (1/2,-1/2) -- (1/2,1/2) -- (-1/2,1/2) -- (-1/2,-1/2);
\draw[xshift=0cm, yshift=-1cm, thick, fill opacity= 0.5, fill=red] (-1/2,-1/2) -- (1/2,-1/2) -- (1/2,1/2) -- (-1/2,1/2) -- (-1/2,-1/2);
\draw[xshift=-1cm, yshift=2cm, thick, fill opacity= 0.5, fill=red] (-1/2,-1/2) -- (1/2,-1/2) -- (1/2,1/2) -- (-1/2,1/2) -- (-1/2,-1/2);
\draw[xshift=-1cm, yshift=0cm, thick, fill opacity= 0.5, fill=red] (-1/2,-1/2) -- (1/2,-1/2) -- (1/2,1/2) -- (-1/2,1/2) -- (-1/2,-1/2);
\draw[xshift=-1cm, yshift=-2cm, thick, fill opacity= 0.5, fill=red] (-1/2,-1/2) -- (1/2,-1/2) -- (1/2,1/2) -- (-1/2,1/2) -- (-1/2,-1/2);
\draw[xshift=2cm, yshift=1cm, thick, fill opacity= 0.5, fill=red] (-1/2,-1/2) -- (1/2,-1/2) -- (1/2,1/2) -- (-1/2,1/2) -- (-1/2,-1/2);
\draw[xshift=2cm, yshift=-1cm, thick, fill opacity= 0.5, fill=red] (-1/2,-1/2) -- (1/2,-1/2) -- (1/2,1/2) -- (-1/2,1/2) -- (-1/2,-1/2);
\draw[xshift=-2cm, yshift=1cm, thick, fill opacity= 0.5, fill=red] (-1/2,-1/2) -- (1/2,-1/2) -- (1/2,1/2) -- (-1/2,1/2) -- (-1/2,-1/2);
\draw[xshift=-2cm, yshift=-1cm, thick, fill opacity= 0.5, fill=red] (-1/2,-1/2) -- (1/2,-1/2) -- (1/2,1/2) -- (-1/2,1/2) -- (-1/2,-1/2);
%Coordinate Axis
\draw[ultra thick, ->] (-3,0) -- (3,0);
\draw[ultra thick, ->] (0,-3) -- (0,3);
%Node
\node at (3.75,2.5) {\large $\xi\in \mathbb{R}^d$};

\begin{scope}[xshift=12.5cm,scale=3/7]
%Radial
\filldraw[fill=blue, fill opacity=0.5] (0,0) arc [radius=0, start angle=90, delta angle=360]
                  -- (0,1) arc [radius=1, start angle=90, delta angle=-360]
                  -- cycle;
\filldraw[fill=red, fill opacity=0.5] (0,1) arc [radius=1, start angle=90, delta angle=360]
                  -- (0,2) arc [radius=2, start angle=90, delta angle=-360]
                  -- cycle;
 \filldraw[fill=blue, fill opacity=0.5] (0,2) arc [radius=2, start angle=90, delta angle=360]
                  -- (0,3) arc [radius=3, start angle=90, delta angle=-360]
                  -- cycle;
  \filldraw[fill=red, fill opacity=0.5] (0,3) arc [radius=3, start angle=90, delta angle=360]
                  -- (0,4) arc [radius=4, start angle=90, delta angle=-360]
                  -- cycle;
\filldraw[fill=blue, fill opacity=0.5] (0,4) arc [radius=4, start angle=90, delta angle=360]
                  -- (0,5) arc [radius=5, start angle=90, delta angle=-360]
                  -- cycle;
\filldraw[fill=red, fill opacity=0.5] (0,5) arc [radius=5, start angle=90, delta angle=360]
                  -- (0,6) arc [radius=6, start angle=90, delta angle=-360]
                  -- cycle;                               
%Axis
\draw[ultra thick, ->] (-7,0) -- (7,0);
\draw[ultra thick, ->] (0,-7) -- (0,7);                  
%Node
\node at (7,7/3*2.5) {\large $\xi \in \mathbb{R}^d$};
\end{scope}
\end{tikzpicture}
\end{center}
\caption*{\small{In the left image, we display a partition of \( \rd \) into unit-scale cubes, which forms the basis of the Wiener randomization. In the right image, we display a partition of \( \rd \) into annuli, which forms the basis of the radial randomization.}}
\caption{Partions of \( \rd \)}
\label{fig:partition_rd}
\end{figure}

\begin{definition}[Annular Multiplier]\label{in:def_annular_multiplier}
Let \( f \in L_x^2(\mathbb{R}^d) \),  \( a > 0 \), and \( \delta \in (0,1) \). Then, we define the operator \( A_{a,\delta} \) by setting
\begin{equation}\label{in:eq_annular_multiplier}
\widehat{A_{a,\delta}f}(\xi) := \chi_{[a,(1+\delta)a)}(\| \xi \|_2) \hat{f}(\xi)~.
\end{equation}
In addition, for any \( 0 < a_1 < a_2 \leq \infty \), we also define the operator \( A_{[a_1,a_2)} \) by setting
\begin{equation*}
\widehat{A_{[a_1,a_2)}f}(\xi) := \chi_{[a_1,a_2)}(\| \xi\|_2) \hat{f}(\xi)~. 
\end{equation*}
\end{definition}
Instead of partitioning \( \mathbb{R}^d\) into unit-scale cubes, the idea of the radial randomization is to decompose \( \mathbb{R}^d\) into thin annuli (see Fig. \ref{fig:partition_rd}).

\begin{definition}[Radial Randomization]
Fix a parameter \( \gamma > 0 \) and let \( \{ g_k \}_{k=0}^\infty \) be a sequence of independent standard real-valued  Gaussians. For any \( f \in L_{\text{rad}}^2(\mathbb{R}^d)\), we define its radial symmetrization by
\begin{equation}\label{in:eq_radial_randomization}
f^\omega(x):=  \sum_{k=0}^{\infty} g_k(w) A_{[k^\gamma,(k+1)^\gamma)} f (x)~. 
\end{equation}
\end{definition}

There exist two natural choices of \( \gamma \): Choosing \( \gamma = 1 \) leads to annuli of unit width, whereas choosing \( \gamma= 1/d  \) leads to annuli of approximately unit volume. \\
We now make a few remarks on the properties of \( f^\omega \). First, since the Fourier transform of \( f^\omega \) is radial, it follows that \( f^\omega \) is radial. Using the same argument as for the Wiener randomization \cite[Lemma 43]{Oh17}, it is easy to see that the radial randomization does not improve the regularity of \( f \). More precisely, if \( s\in \mathbb{R}\) is such that \( f \not \in H_x^s(\rd) \), then \( f^\omega \not \in H_x^s(\rd) \) almost surely. In light of the unboundedness of the ball-multiplier (cf. \cite{Chanillo84,Fefferman71}), it is much harder to prove \(L^p\)-improving properties for the radial randomization than for the Wiener randomization.
The probabilistic Strichartz estimates for the random linear evolution \( \Ht f^\omega \) will be derived from a refined (deterministic) radial Strichartz estimate. In contrast to the Wiener randomization, the radial randomization does not lead to a probabilistic gain of integrability in every non-sharp admissible Strichartz space. Thus, we see a relationship between the geometric structure of the linear evolution and the effects of the randomization, which was also discussed in \cite{CCMNS18}. \\
Let us now formulate the main result of this work. In the following, we restrict the discussion to the dimension \( d=3 \). Let \( (f,g) \in H_{\text{rad}}^s(\rthree) \times H_{\text{rad}}^{s-1}(\rthree) \) be the given (deterministic) initial data. For technical reasons, we split the randomized initial data \( (f^\omega, g^\omega ) \) into low- and high-frequency components. For the high-frequency component, we let 
\begin{equation}\label{in:eq_random_linear}
F^\omega(t,x) = \cos(t|\nabla|) P_{>2^6} f^\omega (x) + \frac{\sin(t|\nabla|)}{|\nabla|} P_{>2^6} g^\omega(x)
 \end{equation}
 be the random and rough linear evolution. Next, we decompose the solution \( u \) of the energy critical NLW into the linear component \( F^\omega \) and a nonlinear component \( v \), i.e., \( u= F^\omega +v \). Then, the nonlinear component solves the initial value problem 
 \begin{equation}\label{intro:eq_forced_nlw}
 \begin{cases}
-\partial_{tt} v + \Delta v =(v+F^\omega)^5~,\\
v(0,x)= P_{\leq 2^6} f^\omega, \qquad \partial_t v(0,x) = P_{\leq 2^6} g^\omega~. 
\end{cases}
 \end{equation}
Note that the initial data in \eqref{intro:eq_forced_nlw} almost surely lies in the energy-space \( \dot{H}_x^1(\rthree)\times L_x^2(\rthree) \). The above decomposition into a linear and nonlinear part is often called
the Da Prato-Debussche trick \cite{PD02}. In the following, we analyze the solution \( v \) of the forced equation \eqref{intro:eq_forced_nlw}. Since \( u=F^\omega+ v\), any statement about \( v \) can easily be translated into a statement about \( u \). 
\begin{thm}[Almost sure scattering]\label{thm:main}
Let \( (f,g) \in H_{\text{rad}}^{s}(\rthree) \times H_{\text{rad}}^{s-1}(\rthree) \),  let \( 0 < \gamma \leq 1 \), and let \( \max(1-\frac{1}{12\gamma},0)< s < 1\). Then, almost surely there exists a global solution \( v \) of \eqref{intro:eq_forced_nlw} such that 
\begin{equation*}
v \in C_t^0 \dot{H}_x^1(\mathbb{R}\times \rthree) ~ {\mathsmaller\bigcap} ~ L_t^5L_x^{10}(\mathbb{R}\times \mathbb{R}^3),\qquad \partial_t v \in C_t^0 L_x^2(\mathbb{R}\times \rthree)~. 
\end{equation*}
Furthermore, there exist scattering states \( (v_0^\pm,v_1^\pm )\in \dot{H}_x^1(\rthree)\times L_x^2(\rthree) \) such that, if \( w^\pm(t) \) are the solutions to the linear wave equation with initial data \( (v_0^\pm, v_1^\pm ) \), we have
\begin{equation*}
\| (v(t) - w^\pm(t), \partial_t v(t)- \partial_t w^\pm(t) ) \|_{\dot{H}_x^1(\rthree)\times L_x^2(\rthree)} \rightarrow 0 \quad \text{as} \quad t \rightarrow \pm \infty~. 
\end{equation*}
\end{thm}
~\\\vspace{-3ex}

We remark that the restriction on \( s \) and the range for \( \gamma \) are not optimal, see e.g. Lemma \ref{bootstrap:lem_auxiliary_norms} and Remark \ref{major:rem_regularity}. For any \( (u_0, u_1 )\in \dot{H}_{\rad}^1(\rthree)\times L_{\rad}^2(\rthree) \), we can also replace the initial data in \eqref{intro:eq_forced_nlw} by \( (u_0 + P_{\leq 2^6} f^\omega, u_1 + P_{\leq 2^6} g^\omega ) \) . This implies the stability of the scattering mechanism of \eqref{in:eq_nlw} under random radial pertubations. \\
By using the deterministic theory and a perturbation theorem, the proof of Theorem \ref{thm:main} reduces to an \textit{a priori} energy bound on \( v \), see \cite{BOP2015,DLM17,Pocovnicu17}. 
We will discuss this reduction in Section \ref{section:lwp}. For now, let us simply state the \textit{a priori} energy bound as a separate theorem. 

\begin{thm}[\textit{A  priori} energy bound]\label{thm:a_priori_bound}
Let \( (f,g) \in H_{\text{rad}}^{s}(\rthree) \times H_{\text{rad}}^{s-1}(\rthree) \),  let \( 0 < \gamma \leq 1 \), and let \( \max(1-\frac{1}{12\gamma},0)< s < 1\). Assume that almost surely there exists a solution \( v \) of \eqref{intro:eq_forced_nlw} with some maximal time interval of existence \( I \). Then, we have that almost surely 
\begin{equation}\label{intro:eq_a_priori_bound}
\sup_{t\in I} E[v](t) < \infty ~. 
\end{equation}
\end{thm}

We now sketch the idea behind the proof of the \textit{a priori} energy bound, which relies on a bootstrap argument. Let us fix a time interval \( I=[a,b] \subseteq \mathbb{R} \). We want to bound the
energy increment \( E[v](b)-E[v](a) \) by the maximal energy \( \mathcal{E} \) of \( v \) on \( I \). We will see that the main error term in the energy increment is given by 
\begin{equation}\label{in:eq_heuristic_error}
 \int_{I} \int_{\rthree} F^\omega v^4 \partial_t v \dx \dt ~.
\end{equation}
In the following discussion, we argue heuristically and ignore all other error terms. Using a Littlewood-Paley decomposition, we may assume that the linear evolution \( F^\omega \) is localized to frequency \( \sim N \).  In dimension \( d=4 \), Dodson, Lührmann and Mendelson \cite{DLM17} used the Morawetz estimate to control the energy increment. Following their idea, we may assume under a bootstrap hypothesis that 
\begin{equation*}
\| |x|^{-\frac{1}{6}} v \|_{L_{t,x}^6(\sI)}^6 \lesssim \mathcal{E}~. 
\end{equation*}
After directly applying the Morawetz estimate to \eqref{in:eq_heuristic_error}, the best possible bound is \( \sim (E^{\frac{1}{6}})^4 E^{\frac{1}{2}} \sim E^{\frac{7}{6}} \). However, this cannot prevent the finite-time blowup of the energy. Following \cite{OP16}, we move the time derivative onto the linear evolution \( F^\omega_N \). First, we write \( \partial_t F^\omega_N = |\nabla|\widetilde{F}^\omega_N \), where \( \widetilde{F}^\omega_N \) is a different solution to the linear wave equation. After neglecting boundary terms, we heuristically rewrite the main error term as 
\begin{equation}\label{in:eq_heuristic_error_2}
 \int_{I} \int_{\rthree}( \partial_t F^\omega_N ) v^5\dx \dt  = \int_{I} \int_{\rthree} (|\nabla|\widetilde{F}_N^\omega) v^5 \dx \dt \sim \int_{I} \int_{\rthree} (|\nabla|^{\frac{1}{2}}\widetilde{F}_N^\omega)~ v^4~ (|\nabla|^{\frac{1}{2}}v) \dx \dt.
\end{equation}
By using the Morawetz term and the energy, we estimate
\begin{align*}
| \int_{I} \int_{\rthree} (|\nabla|^{\frac{1}{2}}\widetilde{F}_N^\omega)~ v^4~ (|\nabla|^{\frac{1}{2}}v) \dx \dt |&\lesssim \| |x|^{\frac{3}{4}} |\nabla|^{\frac{1}{2}} \widetilde{F}_N^\omega \|_{L_t^4 L_x^\infty(\sI)} \| |x|^{-\frac{1}{6}} v \|_{L_{t,x}^6(\sI)}^{\frac{9}{2}} 
\| \nabla v \|_{L_t^\infty L_x^2(\sI)}^{\frac{1}{2}}\\
&\lesssim \| |x|^{\frac{3}{4}} |\nabla|^{\frac{1}{2}} \widetilde{F}^\omega_N \|_{L_t^4 L_x^\infty(\sI)}~ \mathcal{E}~. 
\end{align*}
In this bound, the power of \( \mathcal{E} \) allows us to use a Gronwall-type argument. However, even for smooth and localized initial data, the linear evolution \( |\nabla|^{\frac{1}{2}} \widetilde{F}_N^\omega \) only decays like \( (1+|t|)^{-1} \) and is morally supported near the light cone \( |x|=|t| \). Thus, the norm \(  \| |x|^{\frac{3}{4}} |\nabla|^{\frac{1}{2}} \widetilde{F}_N^\omega \|_{L_t^4 L_x^\infty(\sI)} \) diverges logarithmically as the time interval \( I \) increases. Since the energy yields better decay for \( \nabla v\) than for \( v \) itself, the logarithmic divergence cannot be avoided by placing fewer derivatives on \( v \). Consequently, this argument does not yield global bounds on the energy of \( v \). \\
To overcome the logarithmic divergence, we introduce two additional ingredients. First, since the radial randomization preserves the spherical symmetry of the initial data, the linear evolution \( |\nabla|^{\frac{1}{2}} \widetilde{F}_N^\omega\) is spherically symmetric. Using this, we can decompose the linear evolution into an incoming and outgoing wave, i.e., 
\begin{equation*}
|\nabla|^{\frac{1}{2}} \widetilde{F}_N^\omega = \frac{1}{|x|} \left( W_{\text{in}}[|\nabla|^{\frac{1}{2}} \widetilde{F}_n^\omega](t+|x|) + W_{\text{out}}[|\nabla|^{\frac{1}{2}} \widetilde{F}_n^\omega](t-|x|) \right)~.
\end{equation*}
Second, we use a flux estimate to control the integral of the potential \( v^6 \) on shifted light cones by the energy. We now combine both of these tools by integrating the profiles \( |W_{\text{in}}[|\nabla|^{\frac{1}{2}} \widetilde{F}_n^\omega]|^2(\tau) \)  and \( |W_{\text{out}}[|\nabla|^{\frac{1}{2}} \widetilde{F}_n^\omega]|^2(\tau) \) against the flux estimate on \( t\pm |x|=\tau \). Under a bootstrap hypothesis, we obtain the interaction flux estimate 
\begin{align*}
\int_I \int_{\rthree} |x|^2 ||\nabla|^{\frac{1}{2}} \widetilde{F}_N^\omega|^2 v^6 \dx \dt 
&\lesssim \left( \| W_{\text{in}}[|\nabla|^{\frac{1}{2}} \widetilde{F}_n^\omega](\tau) \|_{L_\tau^2(\mathbb{R})}^2 + \| W_{\text{out}}[|\nabla|^{\frac{1}{2}} \widetilde{F}_n^\omega](\tau) \|_{L_\tau^2(\mathbb{R})}^2 \right)  
\mathcal{E}
\\&\lesssim
 \| ( f^\omega_N,g^\omega_N) \|_{\dot{H}_x^{\frac{1}{2}}\times \dot{H}_x^{-\frac{1}{2}}}^{2} \mathcal{E}~. 
\end{align*}
We have not seen this estimate in the previous literature. It is reminiscent of the interaction Morawetz estimate for the NLS \cite{CKSTT04}, but it controls an interaction between the linear and nonlinear evolution rather than the interaction of the nonlinear evolution with itself. 
We believe that similar interaction estimates may be of interest beyond this work. 
Using the interaction flux estimate, we bound
\begin{align*}
&\left| \int_{I} \int_{\rthree} (|\nabla|^{\frac{1}{2}}\widetilde{F}_N)~ v^4~ (|\nabla|^{\frac{1}{2}}v) \dx \dt \right| \\
&\lesssim \| |x|^{\frac{3}{8}} |\nabla|^{\frac{1}{2}} \widetilde{F}_N^\omega \|_{L_t^\frac{8}{3} L_x^\infty(\sI)}^{\frac{2}{3}} \| |x|^{\frac{1}{3}} (|\nabla|^{\frac{1}{2}} \widetilde{F}_N^\omega)^{\frac{1}{3}}v \|_{L_{t,x}^6(\sI)} \| |x|^{-\frac{1}{6}} v \|_{L_{t,x}^6(\sI)}^{\frac{7}{2}} \| \nabla v \|_{L_t^\infty L_x^2(\sI)}^{\frac{1}{2}} \\
&\lesssim  \| |x|^{\frac{3}{8}} |\nabla|^{\frac{1}{2}} \widetilde{F}_N^\omega \|_{L_t^\frac{8}{3} L_x^\infty(\sI)}^{\frac{2}{3}}  \| (f^\omega_N,g^\omega_N) \|_{\dot{H}_x^{\frac{1}{2}} \times \dot{H}_x^{-\frac{1}{2}}}^{\frac{1}{3}} \mathcal{E}~.
\end{align*}
From the probabilistic Strichartz estimates, we will see that the semi-norm of \( \widetilde{F}^\omega \) 
 scales like \( \dot{H}_x^{\frac{5}{4}} \times \dot{H}_x^{\frac{1}{4}}\) and has a probabilistic gain of \( \frac{1}{8\gamma} \)-derivatives. Thus, we expect the regularity restriction 
\begin{equation*}
s > \tfrac{2}{3} \cdot \left( \tfrac{5}{4} - \tfrac{1}{8\gamma} \right) + \tfrac{1}{3}  \cdot \tfrac{1}{2} = 1 - \tfrac{1}{12\gamma} ~.
\end{equation*}

\paragraph{Outline.}~\\
In Section \ref{sec:prelim}, we review basic facts from harmonic analysis. In Sections 
\ref{section:prob_strichartz} and \ref{sec:decomp}, we study solutions to the radial linear wave equation. First, we prove a refined radial Strichartz estimate which is based on \cite{Sterbenz05}. As a consequence, we obtain 
probabilistic Strichartz estimates for the radial randomization. Then, we discuss the in/out decomposition mentioned above in detail. In Sections \ref{section:lwp} and \ref{section:energy}, we study solutions to the forced nonlinear wave equation \eqref{intro:eq_forced_nlw}. We prove an almost energy conservation law and an approximate Morawetz estimate. Here, we also introduce the novel interaction flux estimate between the linear and nonlinear evolution. In Sections \ref{section:bootstrap} and \ref{section:major}, we set up a bootstrap argument to bound the energy and estimate the error terms. Finally, we prove the main theorem in {\mbox{Section \ref{section:a_priori_bound}.}}

\paragraph{Acknowledgements}~\\
The author thanks his advisor Terence Tao for his invaluable guidance and support. The author also thanks Laura Cladek, Rowan Killip, and Monica Visan for many insightful discussions. 

\section{Notation and preliminaries} \label{sec:prelim}
In this section, we introduce the notation that will be used throughout the rest of this paper. We also recall some basic results from harmonic analysis and prove certain auxiliary lemmas. \\
If \( A \) and \( B \) are two nonnegative quantities, we write \( A \lesssim B \) if there exists an absolute constant \( C > 0 \) such that \( A \leq C B \). We write \( A \sim B \) if \( A \lesssim B \) and \( B \lesssim A \). For a vector \( x \in \rd \), we write \( |x|:= (\sum_{i=1}^d x_i^2 )^{\frac{1}{2}} \). We define the Fourier transform of a Schwartz function \( f \) by setting
\begin{equation*}
\widehat{f}(\xi) := \tfrac{1}{(2\pi)^{\frac{d}{2}}} \int_{\rd} \exp(-ix\xi) f(x) \dx~. 
\end{equation*}
We denote by \( J_\nu(x) \) the Bessel functions of the first kind. Recall that for a spherically symmetric function \( f \) we have 
\begin{equation*}
\widehat{f}(\xi) = |\xi|^{-\frac{d-2}{2}} \int_0^\infty J_{\frac{d-2}{2}}(|\xi|r) f(r) r^{\frac{d}{2}} \dr~. 
\end{equation*}
With a slight abuse of notation, we identify a spherically symmetric function \( f \colon \rd \rightarrow \mathbb{R} \) with a function \( f \colon \mathbb{R}_{>0} \rightarrow \mathbb{R} \). 

\subsection{Littlewood-Paley theory and Sobolev embeddings}
We start this section by defining the Littlewood-Paley operators \( P_L \). 
Let \( \phi \in C^\infty_c(\mathbb{R}^d)\) be a nonnegative radial bump function such that 
\( \phi|_{B(0,1)} \equiv 1 \) and \( \phi_{\mathbb{R}^d\backslash B(0,2)} \equiv 0 \). We set \( \Psi_1(\xi) = \phi(\xi) \) and, for a dyadic \( L > 1 \), we set \( \Psi_L(\xi) = \phi(\frac{\xi}{L}) -\phi(\frac{\xi}{2L}) \). Then, we define the Littlewood-Paley operators \( P_L \) by 
\begin{equation*}
\widehat{P_L f}(\xi) = \Psi_L(\xi) \widehat{f}(\xi)~. 
\end{equation*}
To simplify the notation, we also write \( f_L := P_L f \). 

\begin{lem}[Bernstein Estimate]
For any \( 1 <p_1 \leq p_2 < \infty \) and \(s \geq 0 \), we have the Bernstein inequalities
\begin{alignat}{3}
&\qquad\forall L\geq 1\colon &\quad \| f_L \|_{L_x^{p_2}(\rd)} &\lesssim L^{\frac{d}{p_1}-\frac{d}{p_2}} \| f_L \|_{L_x^{p_1}(\rd)}~, \notag\\
& \qquad \forall L > 1\colon& \quad \| |\nabla|^{\pm s} f_L \|_{L_x^{p_1}(\rd)} &\sim L^{\pm s} \| f_L \|_{L_x^{p_1}(\rd)} \quad~,\notag\\
& \qquad \forall L > 1\colon& \quad \| \nabla f_L \|_{L_x^{p_1}(\rd)} &\sim L \|f_L \|_{L_x^{p_1}(\rd)} \quad~.\label{prelim:eq_bernstein_gradient} \notag
\end{alignat}
\end{lem}

\begin{lem}[{Square-Function Estimate, see \cite[Theorem 8.3]{MuscaluSchlag}}] \label{prelim:lem_square_function}
Let \( 1 < p < \infty \). Then, we have for all \( f \in L_x^p(\rd) \) that 
\begin{equation}
\| f \|_{L_x^p(\rd)} \sim_{d,p} \| f_L \|_{L_x^p \ell_L^2(\rd \times 2^{\mathbb{N}} )}~. 
\end{equation}
\end{lem}

For notational convenience, we use a different function to define a dyadic decomposition in physical space. 
As before, we let \( \chi \in C^{\infty}_c(\mathbb{R}^d) \) be a nonnegative, radial bump function such that \( \chi|_{B(0,1)} \equiv 1 \) and \( \chi|_{\mathbb{R}^d\backslash B(0,2)} \equiv 0 \). We also assume that \( \chi \) is radially non-increasing. We set \( \chi_1 := \chi \), and for any dyadic \( J> 1\), we set \( \chi_J(x) := \chi(\frac{x}{J}) - \chi(\frac{2x}{J}) \). Thus, the family \( \{ \chi_J \}_{J\geq 1} \) defines a partition of unity adapted to dyadic annuli. Furthermore, we let \( \widetilde{\chi_J} \) be a slightly fattened version of \( \chi_J \). 
\begin{lem}[Mismatch Estimate] 
Let \( L,J,K \in 2^{\mathbb{N}_0} \). Furthermore, we assume that the separation condition \( \frac{J}{K} + \frac{K}{J} \geq 2^5 \) holds. Then, we have for all \( 1 \leq r \leq \infty \) that 
\begin{equation}\label{prelim:eq_localized_kernel}
\| \chi_J P_L \chi_K \|_{L_x^r(\mathbb{R}^d)\rightarrow L_x^r(\mathbb{R}^d)} \lesssim_{M}  (LJK)^{-M} \qquad \text{for all } M>0~. 
\end{equation}
\end{lem}
We follow the argument in \cite[Lemma 5.10]{DLM18}, which treats the case \( L=1 \). 
\begin{proof}
Let \( f\in L_x^r(\mathbb{R}^d) \) be arbitrary. Let \( \varphi \) be a suitable bump function on the annulus \( |x|\sim 1 \). Using the separation condition, it holds that
\begin{align*}
\chi_J P_L \chi_K f(x) &= \chi_J(x) L^d \int_{\mathbb{R}^d} \check\Psi(L(x-y)) \chi_K(y) f(y)\dy \\
& =\chi_J(x) L^d \int_{\mathbb{R}^d}\check \Psi(L(x-y)) \varphi(\max(J,K)^{-1} (x-y)) \chi_K(y) f(y)\dy~. 
\end{align*}
From Young's inequality, it follows that 
\begin{equation*}
\| \chi_J P_L \chi_K f \|_{L_x^r(\rd)} \leq \| L^d\check \Psi(Lx) \varphi(\max(J,K)^{-1}x) \|_{L_x^1(\rd)} \| f \|_{L_x^r(\rd)}~. 
\end{equation*}
Next, we estimate
\begin{align*}
 &\| L^d \check\Psi(Lx) \varphi(\max(J,K)^{-1}x) \|_{L_x^1} \\
 &= L^d \int_{\mathbb{R}^d} |\widecheck\Psi(Lx)| \varphi(\max(J,K)^{-1}x) \dx \\
 &= \int_{\mathbb{R}^d} |\widecheck\Psi(x)| \varphi(L^{-1}\max(J,K)^{-1}x) \dx \\
 &= \int_{|x|\sim L \max(J,K)} |\widecheck\Psi(x)| \dx \\
 &\lesssim_M (L\max(J,K))^{-M}~. 
\end{align*}
\end{proof}

\begin{lem}[Bernstein-type estimate]\label{prelim:lem_weighted_sobolev}
Let \( L \in 2^{\mathbb{N}_0} \), \( 1< p < \infty \), and \( \alpha> 0 \). Then, we have that 
\begin{equation}
\| \langle x \rangle^{-\alpha} P_L f \|_{L_x^p(\rd)} \lesssim L^{-1} \| \langle x \rangle^{-\alpha} \nabla f \|_{L_x^p(\rd)} + L^{-1} \| \langle x \rangle^{-\alpha-1} f \|_{L_x^p(\rd)}~. 
\end{equation}
\end{lem}
By iterating this inequality, we could further decrease the weight in the term \( \| \langle x\rangle^{-\alpha-1} f \|_{L_x^p} \). 
\begin{proof}
The proof is based on a  dyadic decomposition, the localized kernel estimate \eqref{prelim:eq_localized_kernel}, and the standard Bernstein estimate. We have that 
\begin{align}
\| \langle x \rangle^{-\alpha} P_L f \|_{L_x^p(\rd)}^p 
&\lesssim \sum_{J\geq 1}^\infty J^{-\alpha p} \| \chi_J P_L f \|_{L_x^p(\rd)}^p \notag  \\
&\lesssim \sum_{J\geq 1}^\infty J^{-\alpha p } \| \chi_J P_L \widetilde{\chi_J} f \|_{L_x^p(\rd)}^p 
+ \sum_{J\geq 1}^\infty J^{-\alpha p}\left( \sum_{K\colon K \not \sim J}\| \chi_J P_L \chi_K f \|_{L_x^p(\rd)} \right)^p~.  \label{prelim:eq_weighted_sobolev_proof_1}
\end{align}
We now estimate the first summand in \eqref{prelim:eq_weighted_sobolev_proof_1}. Using the Bernstein estimate, we have that 
\begin{align*}
 &\sum_{J\geq 1}^\infty J^{-\alpha p } \| \chi_J P_L \widetilde{\chi_J} f \|_{L_x^p(\rd)}^p \\
 &\leq  \sum_{J\geq 1}^\infty J^{-\alpha p } \| P_L \widetilde{\chi_J} f \|_{L_x^p(\rd)}^p \\
 &\lesssim \sum_{J \geq 1}^{\infty} J^{-\alpha p} L^{-p} \| \nabla (\widetilde{\chi_J} f) \|_{L_x^p(\rd)}^p \\
  &\lesssim \sum_{J \geq 1}^{\infty} J^{-\alpha p} L^{-p} \| \widetilde{\chi_J} \nabla f \|_{L_x^p(\rd)}^p 
  +  \sum_{J \geq 1}^{\infty} J^{-\alpha p} L^{-p} \| \nabla(\widetilde{\chi_J}) f \|_{L_x^p(\rd)}^p \\
   &\lesssim \sum_{J \geq 1}^{\infty} J^{-\alpha p} L^{-p} \| \nabla f \|_{L_x^p(|x|\sim J)}^p 
  +  \sum_{J \geq 1}^{\infty} J^{-(\alpha+1) p} L^{-p} \| f \|_{L_x^p(|x|\sim J)}^p \\
  &\lesssim  L^{-p} \| \langle x \rangle^{-\alpha} \nabla f \|_{L_x^p(\rd)}^p + L^{-p} \| \langle x \rangle^{-\alpha-1} f \|_{L_x^p(\rd)}^p~
\end{align*}
Thus, it remains to estimate the second summand in \eqref{prelim:eq_weighted_sobolev_proof_1}. Using \eqref{prelim:eq_localized_kernel} and choosing \( M>0 \) large, we have that
\begin{align*}
&\sum_{J\geq 1}^\infty J^{-\alpha p} \left( \sum_{K\colon K \not \sim J}\| \chi_J P_L \chi_K f \|_{L_x^p(\rd)} \right)^p \\
&\lesssim \sum_{J \geq 1}^{\infty} J^{-\alpha p} \left( \sum_{K \colon K \not \sim J} (JKL)^{-(M+\alpha+1)  } \| \widetilde{\chi_K} f \|_{L_x^p(\rd)} \right)^p \\
&\lesssim L^{-(M+\alpha+1)p } \sum_{J\geq 1}^{\infty} J^{-(M+2\alpha +1 )p} \left( \sum_{K\geq 1} K^{-M} \right)^p  \| \langle x \rangle^{-\alpha-1} f \|_{L_x^p(\rd)}^p\\
&\lesssim L^{-p}  \| \langle x \rangle^{-\alpha-1} f \|_{L_x^p(\rd)}^p~.
\end{align*}
\end{proof}

In Section \ref{sec:morawetz_error}, we will use a Littlewood-Paley decomposition in an error term coming from the Morawetz estimate. To control this error, we will need the following estimate for the Morawetz weight \( x/|x| \). 
\begin{lem}\label{prelim:lem_lwp_pointwise}
Let \( L > 1 \) and let \( d \geq 2 \). Then, we have that 
\begin{equation}
\left|P_L\left( \frac{x}{|x|} \right) \right|\lesssim \frac{1}{L|x|}~. 
\end{equation}
\end{lem}
\begin{proof} Let \( j =1,\hdots,d \). It holds that 
\begin{align*}
|P_L(\frac{x_j}{|x|}) | &= L^d \left| \int_{\rd}  \widecheck{\Psi}(Ly) \frac{x_j-y_j}{|x-y|}  \dy \right| \\
					&=   L^d \left| \int_{\rd}  \widecheck{\Psi}(Ly) \left( \frac{x_j-y_j}{|x-y|}- \frac{x_j}{|x|} \right) \dy \right|  \\
					&\leq L^d \int_{\rd} |\widecheck{\Psi}(Ly) |\left|\frac{x_j(|x|-|x-y|)-y_j |x|}{|x-y||x|} \right| \dy \\
					&\leq L^d \int_{\rd} |\widecheck{\Psi}(Ly)| \frac{|y|}{|x-y|} \dy \\
					&\leq \int_{\rd} |\widecheck{\Psi}(y)| \frac{|y|}{|Lx-y|} \dy~. 
\end{align*}
Using the rapid decay of \( \widecheck{\Psi} \), the estimate then follows by splitting the integral into the regions \( |y|\leq \tfrac{L|x|}{2}, |y| \sim L|x|, \) and \( |y|\geq 2L|x| \). 

\end{proof}

In addition to the standard Sobolev embedding, we will also rely on the following weighted Sobolev embedding for radial functions. 
\begin{prop}[{Radial Sobolev Embedding, see \cite[Remark 2.1]{DD16} and \cite{DDD11}}]\label{prelim:prop_radial_sobolev}
Let \( d \geq 1 \), \( 0 < s < d \), \( 1<p < \infty \), \( \alpha < \frac{d}{p^\prime} \), \( \beta >-\frac{d}{q} \), \( \alpha - \beta \geq (d-1)(\frac{1}{q}-\frac{1}{p}) \), and 
\( \frac{1}{q}=\frac{1}{p} + \frac{\alpha-\beta-s}{d} \). If \( p \leq q < \infty \), then the inequality 
\begin{equation}
\| |x|^{\beta} f \|_{L_x^q} \lesssim \| |x|^\alpha |\nabla|^s f \|_{L_x^p} 
\end{equation}
holds for all radially symmetric \( f \). If \( q = \infty \), the result holds provided that \( \alpha - \beta > (d-1) (\frac{1}{q}-\frac{1}{p}) \). 
\end{prop}

\subsection{{Calderón-Zygmund theory}}

In order to use weighted estimates, we introduce some basic Calderón-Zygmund theory. 
\begin{definition}[{\cite[Section \RNum{5}]{Stein93}}]
Let \( w \in L^1_{\loc}(\rd) \) be nonnegative. 
For \( 1 < p < \infty \), we say that \( w \) satisfies the \( A_p \)-condition if 
\begin{equation}\label{prelim:eq_ap_condition}
\sup_{B=B_r(x)} \left( \frac{1}{|B|} \int_B w \dy \right) ~  \left( \frac{1}{|B|} \int_B w^{-\frac{p^\prime}{p}} \dy \right)^{\frac{p}{p^\prime}} < \infty ~. 
\end{equation}
\end{definition}
~\\
The following well-known criterion for power weights can be proven by a simple computation. 
\begin{lem}[{\cite[Section \RNum{5}.6]{Stein93}}] 
Let \( w= |x|^\alpha \) and let \( 1 < p < \infty \). Then \( w \) satisfies the \( A_p \)-condition if and only if 
\begin{equation*}
 -d < \alpha < d(p-1) ~.
 \end{equation*}
 \end{lem}
 
 The following proposition is a consequence of \cite[Theorem 7.21]{MuscaluSchlag} and the proof of \cite[Theorem 8.2]{MuscaluSchlag}. We also refer the reader to \cite[p.205]{Stein93}. 

\begin{prop}[Mikhlin-multiplier theorem]\label{prelim:prop_mikhlin}
Let \( m \colon \rd \backslash \{ 0 \} \rightarrow \mathbb{C} \) be a smooth function. Assume that \( m \) satisfies for any multiindex \( \gamma \) of length \( |\gamma| \leq d+2 \)  
\begin{equation*}
|\partial^\gamma m(\xi)|\leq B |\xi|^{-|\gamma|}~.
\end{equation*} 
Let \( m(\nabla/i) \) be the associated Fourier multiplier and let \( 1 < p <\infty \). For any \( A_p\)-weight \( w \), there exists a constant \( C \) depending only on \( d, p\), and the supremum in \eqref{prelim:eq_ap_condition}, such that 
\begin{equation*}
\| m(\nabla/i) f \|_{L^p(w\dx)} \leq C B \| f \|_{L^p(w\dx)} \quad \forall f \in \mathcal{S}(\rd)~. 
\end{equation*}
\end{prop}
\begin{rem}
We will apply Proposition \ref{prelim:prop_mikhlin} to the Riesz multipliers \( m_j(\xi) = \frac{\xi_j}{|\xi|} \) and to the Littlewood-Paley multipliers \( \Psi_L(\xi) \). 
\end{rem}

\section{Probabilistic Strichartz estimates}\label{section:prob_strichartz}
In this section, we derive probabilistic Strichartz estimates for the radial randomization. 
For the Wiener randomization, there exist two different methods for proving probabilistic Strichartz estimates. \\
The first method relies on Bernstein-type inequalities for the multipliers \( f \mapsto \psi(\nabla/i-k) f \). After using Khintchine's inequality to decouple the individual atoms of the randomization, the \( L_x^p \)-improving properties of the multiplier are used to move from a space \(  L_t^q L_x^{p_{\text{hi}}} \) into  a space \( L_t^q L_x^{p_{\text{lo}}} \). Then, one applies the usual Strichartz estimate to control the evolution in \( L_t^q L_x^{p_{\text{lo}}} \), which depends more favorably on the regularity of the initial data. For example, this method has been used in \cite{BOP2015,BOP2014,BOP17,BOP17b,KMV17,LM13}. \\
The second method relies on refined Strichartz inequalities. Here, the frequency localization is used explicitly to derive improved Strichartz estimates. To mention one example, the refined Strichartz estimate in \cite{KT99} is based on a new \( L_x^1\rightarrow L_x^\infty\)-dispersive decay estimate. In the probabilistic context, this approach was first used in \cite{DLM17}. \\
For the radial randomization, the multipliers are of the form \( f \mapsto A_{a,\delta} f \). In a celebrated paper \cite{Fefferman71}, Fefferman proved that the annular Fourier multipliers in dimension \( d \geq 2 \) are bounded on \( L^p \) if and only if \( p=2 \). However, if we restrict to radial functions, then the annular Fourier multipliers are bounded on \( L^p \) for all \( 2d/(d+1)< p < 2d/(d-1)\), see \cite{Chanillo84}. Using Young's inequality, it is also possible to prove \( L_x^1 \rightarrow L_x^p \) bounds for \( p> 2d/(d+1) \). From interpolation and duality, one can then obtain  the strong-type diagram for the annular Fourier-multipliers on radial functions. However, the dependence of the operator norm on the normalized width \( \delta \) is rather complicated, and the resulting Strichartz estimates are non-optimal. 
Instead of using the Bernstein-based method, we therefore prove a new refined Strichartz estimate for radial initial data. As in previous works, we can then use Khintchine's inequality to obtain probabilistic Strichartz estimates. 

%Scaling factor
\newcommand{\sfac}{1/3}
\begin{figure}[t!]
\begin{center}
\begin{tikzpicture}[scale=8]
%Draw coordinate axis
\draw[very thick,->] (0,0,0) -- (0.8,0,0);
\draw[very thick,->] (0,-\sfac*1.25,0) -- (0,\sfac*3/2,0);
\draw[very thick,->] (0,0,0) -- (0,0,-0.8);
\node[below right] at (0.8,0,0) {\large $\frac{1}{q}$};
\node[above left] at (-0.02,\sfac*3/2,0) {\large $\alpha$};
\node[above left] at (0.04,0,-0.8) {\large $\frac{1}{p}$};
\node[above right] at (1/2,0,-1/2) { $\big(\frac{1}{2},\frac{1}{2},0\big)$};
%Axis legend
\node[below] at  (1/2, -0.02,0) { \large $\frac{1}{2}$} ; %q-axis
\draw[thick] (0-0.01,\sfac*1,0) node[left] {  $1~$} -- (0+0.01,\sfac*1,0); %alpha-axis
\draw[thick] (0-0.01,-\sfac*1,0) node[left] { -$1~$} -- (0+0.01,-\sfac*1,0); %alpha-axis
\draw[thick] (0.04,0.01,-1/2+0.05) node[left] { \large $\frac{1}{2}~~~$} -- (0.04,0.01,-1/2+0.05); %paxis
\draw[thick] (-0.007,0.007,-1/2) -- (0.007,-0.007,-1/2);
\draw[thick] (1/2,0.01,0) -- (1/2,-0.01,0);

%black anchor
\draw[thick, gray, dashed] (1/2,-0.01,-1/2) -- (1/2,-\sfac*1/2+0.01,-1/2);
\ballblack{1/2}{0}{-1/2}

%Polygon
\draw[thick, draw opacity=0, fill=green, fill opacity=0] (0,0,0)-- (0,\sfac*1,0) -- (0,0,-1/2) -- (0,-\sfac*3/2,-1/2) -- (0,0,0);  %abcd
\draw[thick, draw opacity=0, fill=green, fill opacity=0.1,shade] (0,0,0)-- (0,\sfac*1,0) -- (1/2,\sfac*1/2,0) -- ( 1/2,0, 0) -- (0,0,0);  %abfe
\draw[thick,  draw opacity=0,fill= green, fill opacity =0,shade]   (0,0,-1/2)-- (0,-\sfac*3/2,-1/2) -- (1/2,-\sfac*3/2,-1/2) -- (1/2,-\sfac*1/2,-1/2)--(0,0,-1/2);				%cdgh
\draw[thick, draw opacity=0, fill=green, fill opacity=0,shade]	(0,\sfac*1,0) -- (0,0,-1/2) -- 	(1/2,-\sfac*1/2,-1/2)--(1/2,\sfac*1/2,0) -- (0,\sfac*1,0);				%bchf 
\draw[thick,  draw opacity=0,fill=green, fill opacity =0.1,shade] 	(0,0,0) -- (0,-\sfac*3/2,-1/2) -- 	(1/2,-\sfac*3/2,-1/2)  -- ( 1/2,0, 0) -- (0,0,0);	%adge
\draw[thick,  draw opacity=0,fill=green, fill opacity =0.1,shade]	( 1/2,0, 0)-- (1/2,\sfac*1/2,0)	-- (1/2,-\sfac*1/2,-1/2)-- (1/2,-\sfac*3/2,-1/2)   -- (1/2,0,0);%efhg
%Connecting lines
\draw[very thick, red!60] (0,\sfac*1,0) -- (1/2,\sfac*1/2,0); %bf
\draw[very thick,dashed ] (0,\sfac,0) -- (0,0,-1/2); %bc
\draw[very thick, dashed, red!60] (0,0,-1/2)--(1/2,-\sfac*1/2,-1/2); %ch 
\draw[very thick, red!60] (1/2,-\sfac*1/2,-1/2)--(1/2,\sfac*1/2,0); %hf 
\draw[very thick] (1/2,\sfac*1/2,0) -- ( 1/2,0, 0); %fe
\draw[very thick] (1/2,-\sfac*1/2,-1/2)-- (1/2,-\sfac*3/2,-1/2); %hg
\draw[very thick, red!60] (1/2,-\sfac*3/2,-1/2)  -- ( 1/2,0, 0); %ge 
\draw[very thick, red!60]  (0,-\sfac*3/2,-1/2) -- (1/2,-\sfac*3/2,-1/2); %dg 
\draw[very thick, red!60] (0,0,0) -- (0,-\sfac*3/2,-1/2); %ad
\draw[very thick, dashed] (0,0,-1/2) -- (0,-\sfac*3/2,-1/2); %cd

%Points
\ballpno{0}{0}{0} 		%a
\ballpno{0}{\sfac}{0}       %b
\ballpno{0}{0}{-1/2}	       %c
\ballred{0}{-\sfac*3/2}{-1/2}   %d
\ballpno{1/2}{0}{0}   %e
\ballred{1/2}{\sfac*1/2}{0}			%f
\ballred{1/2}{-\sfac*1/2}{-1/2}			%h
\ballred{1/2}{-\sfac*3/2}{-1/2}			%g

\end{tikzpicture}
~\\\vspace{1ex}
\caption*{\small{We display the radial Strichartz estimate from Proposition \ref{prob:prop_refined_strichartz}. The true endpoint estimates correspond to either green spheres or black lines, whereas the false endpoint estimates  correspond to either red spheres or red lines. The black sphere at \( (1/2,1/2,0) \) serves as a visual aid. }}
\caption{Weighted Radial Strichartz Estimate in \( d=3 \).}
\label{fig:weighted_strichartz}
\end{center}
\end{figure}

\begin{prop}[Refined Radial Strichartz Estimate]\label{prob:prop_refined_strichartz}
Let \( f \in \ltworad(\mathbb{R}^d) \). Let \( 0 < \delta \leq 1 \) and assume that there exists an interval \( I \subseteq [\frac{1}{2},2] \) such that  \( |I|\leq \delta \) and \( \supp \hat{f} \subseteq \{ \xi\colon \| \xi \|_2 \in I \}\). Then, we have that 
\begin{equation}\label{prob:eq_refined_strichartz}
\| |x|^\alpha \Ht f \|_{\ltx{q}{p}(\mathbb{R}\times \rd)} \lesssim_{\alpha,q,p} \delta^{\frac{1}{2}-\frac{1}{\min(p,q)}} \| f \|_{\ltwox(\rd)} 
\end{equation}
as long as 
\begin{align}
-\frac{d}{p} &< \alpha < (d-1) \left( \frac{1}{2}-\frac{1}{p} \right) - \frac{1}{q} \qquad &\text{if} ~~ 2 \leq q,p < \infty \label{prob:eq_refined_strichartz_cond_1}\\
-\frac{d}{p} &< \alpha \leq (d-1) \left( \frac{1}{2}-\frac{1}{p} \right)  \qquad &\text{if} ~~ q=\infty, 2 \leq p < \infty\label{prob:eq_refined_strichartz_cond_2} \\
0 &\leq \alpha <  \frac{d-1}{2} - \frac{1}{q} \qquad &\text{if} ~~ 2 \leq q < \infty, p=\infty \label{prob:eq_refined_strichartz_cond_3} \\
0 & \leq \alpha \leq \frac{d-1}{2} \qquad &\text{if} ~~ q=p=\infty~.\label{prob:eq_refined_strichartz_cond_4}
\end{align}
\end{prop}
The estimates of Proposition \ref{prob:prop_refined_strichartz} can be visualized using a ``Strichartz game room'', see Figure \ref{fig:weighted_strichartz}. 
Proposition \ref{prob:prop_refined_strichartz} is a refinement of \cite[Theorem 1.5]{JWY12} and \cite[Proposition 1.2]{Sterbenz05}, and we follow their argument closely. 
We remark that the corresponding Strichartz estimate for non-frequency localized functions \cite[Theorem 1.5]{JWY12} may fail for some of the endpoints above.

\begin{proof} By time-reflection symmetry, it suffices to treat the operator \( \Htp \). 
Recall that we denote by \( J_\nu \) the Bessel functions of the first kind. For any radial function \( f \in \ltworad(\rd) \), we identify \( \hat{f} \) with a function \( \hat{f}\colon \mathbb{R}_{>0} \rightarrow \mathbb{R} \). Then, it holds that 
\begin{equation}
\Htp f (r)= r^{-\frac{d-2}{2}} \int_0^\infty \exp(it\rho) J_{\frac{d-2}{2}}(r\rho) \hat{f}(\rho) \rho^{\frac{d}{2}} \drho 
\end{equation}
Inserting the known asymptotics for Bessel functions (cf. \cite{JWY12}), we may estimate 
\begin{equation}\label{prob:eq_proof_refined_full}
(1+r)^{-\frac{d-1}{2}} \int_{0}^{2\pi} \exp( i(t\pm r) \rho)  m(r;\rho) \phi_{(\frac{1}{4},4)}(\rho) \hat{f}(\rho) \drho ~. 
\end{equation}
Here, \( \phi_{(\frac{1}{4},4)} \) is a smooth cutoff-function that equals \( 1 \) on \( [1/2,2] \) and is supported on \( [1/4,4] \), and \( m(r;\rho) \)  is a smooth function that satisfies \( |\partial_\rho^j m(r;\rho) | \lesssim_j 1 \) for all \( j\geq 0 \). Since \( \supp\hat{f} \subseteq [\frac{1}{2}, 2 ] \), we may write 
\begin{equation}\label{prob:eq_proof_refined_fourier_series}
\hat{f}(\rho) = \sum_{k\in \mathbb{Z}}c_k \exp(ik \rho), \qquad \text{where} ~ c_k = \frac{1}{2\pi} \int_{0}^{2\pi} \exp(-i k \rho ) \hat{f}(\rho) \drho~. 
\end{equation}
Inserting \eqref{prob:eq_proof_refined_fourier_series} into \eqref{prob:eq_proof_refined_full}, we have to bound 
\begin{equation}
\sum_{k\in \mathbb{Z}} (1+r)^{-\frac{d-1}{2}} c_k \int_{0}^{2\pi} \exp(i( t \pm r + k )\rho ) m(r;\rho) \phi_{(\frac{1}{4},4)}(\rho) \mathrm{d}\rho~. 
\end{equation}
Integrating by parts \( 2M \)-times, we have that
\begin{equation*}
\left| \int_{0}^{2\pi} \exp(i( t \pm r + k )) m(r;\rho) \phi_{(\frac{1}{4},4)}(\rho) \mathrm{d}\rho \right| \lesssim_M (1+ |t\pm r + k|)^{-2M} ~. 
\end{equation*}
Therefore, we obtain that \begin{align}
&\| |x|^\alpha \Htp f \|_{L_x^p(\rd)}\notag\\
 &\lesssim \| (1+r)^{-\frac{d-1}{2}} r^{\alpha} r^{\frac{d-1}{p}}  (1+|t+ k \pm r|)^{-2M}  c_k \|_{\lrk{p}{1}(\mathbb{R}_{>0}\times \mathbb{Z})} \notag\\
&\lesssim \| (1+r)^{-\frac{d-1}{2}} r^{\alpha+\frac{d-1}{p}}  (1+|t+ k \pm r|)^{-M}  c_k \|_{\lrk{p}{p}(\mathbb{R}_{>0}\times \mathbb{Z})} ~,
\label{prob:eq_proof_refined_reuse}
\end{align}
where we have used Hölder's inequality in the \(k\)-variable.  Since \( \alpha + \frac{d-1}{p} > -\frac{1}{p} \) if \( 2\leq p <\infty \), or \( \alpha \geq 0 \) if \( p =\infty \), we obtain for sufficiently large \( M \) that 
\begin{equation*}
\| (1+r)^{-\frac{d-1}{2}} r^{\alpha+\frac{d-1}{p}}  (1+|t+ k \pm r|)^{-M}   \|_{L_r^p(\mathbb{R}_{>0}) }\lesssim 
(1+|t+k|)^{-\frac{d-1}{2}} |t+k|^{\alpha+ \frac{d-1}{p}} ~. 
\end{equation*}
From the embedding \( \ell_k^{\min(p,q)} \hookrightarrow \ell_k^p  \) and Minkowski's integral inequality, we obtain that 
\begin{align}
&\| |x|^\alpha \Htp f \|_{\ltx{q}{p}(\mathbb{R} \times \rd)} \notag \\
&\lesssim \| (1+|t+k|)^{-\frac{d-1}{2}} |t+k|^{\alpha+ \frac{d-1}{p}} c_k \|_{L_t^q \ell_k^p(\mathbb{R}\times \mathbb{Z})} \notag \\
& \lesssim \| (1+|t+k|)^{-\frac{d-1}{2}} |t+k|^{\alpha+\frac{d-1}{p}} c_k \|_{\ell_k^{\min(p,q)} L_t^q(\mathbb{Z}\times \mathbb{R})} \notag \\ 
&\lesssim \| c_k \|_{\ell_k^{\min(p,q)}(\mathbb{Z}) } \label{prob:eq_proof_refined_ck_bound}~. 
\end{align}
From Plancherell's theorem and the support condition on \( \hat{f} \), we have that 
\begin{equation*}
\| c_k \|_{\ell_k^2(\mathbb{Z})}^2 = \frac{1}{2\pi} \int_0^{2\pi} |\hat{f}(\rho)|^2 \drho \sim \| f \|_{L_x^2(\rd)}^2~. 
\end{equation*}
Furthermore, since \( \supp \hat{f} \) is contained in an interval of size \( \leq \delta\), we have that 
\begin{equation*}
\| c_k \|_{\ell_k^\infty(\mathbb{Z})}\leq \frac{1}{2\pi} \int_I | \hat{f}(\rho)|  \drho~\lesssim \delta^{\frac{1}{2}} \| f \|_{L_x^2(\rd)} ~. 
\end{equation*}
Then  \eqref{prob:eq_refined_strichartz} follows from \eqref{prob:eq_proof_refined_ck_bound} and Hölder's inequality. 
\end{proof}

\begin{rem} We note that there is no \( \delta \)-gain for \( q=2 \). For instance, this follows from a non-stationary phase argument by choosing \( f \) as the inverse Fourier transform of \( \chi_{[1,1+\delta]}(|\xi|) \). As a consequence, we obtain no probabilistic gain for Strichartz estimates with parameter \( q=2 \), see Lemma  \ref{prob:lem_prob_strichartz}. This indicates that the spherical symmetry imposes restrictions on the randomized linear evolutions. We therefore view the radial randomization as a modest step towards probabilistic treatments of the geometric equations discussed in \cite{CCMNS18}. 
\end{rem}

\begin{cor}\label{prob:cor_refined_strichartz_scaled}
Let \( f \in \ltworad(\mathbb{R}^d)\) and \( A_{a,\delta} \) as in \eqref{in:eq_annular_multiplier} with \( a \sim N \). If \( \alpha,p\), and \( q \) satisfy \eqref{prob:eq_refined_strichartz_cond_1}-\eqref{prob:eq_refined_strichartz_cond_4}, then
\begin{equation}\label{prob:eq_refined_strichartz_scaled}
\| |x|^\alpha \Ht A_{a,\delta} f \|_{L_t^qL_x^p} \lesssim N^{\frac{d}{2}-\alpha - \frac{1}{q} - \frac{d}{p}} 
~ \delta^{\frac{1}{2}- \frac{1}{\min(p,q)}}~  \| A_{a,\delta} f \|_{L_x^2}~. 
\end{equation}
\end{cor}
\begin{proof}
For any \( g \in \ltworad(\mathbb{R}^d) \), we have that
\begin{equation*}
A_{a,\delta} g(x) = \left( A_{\frac{a}{N},\delta} \left( g\left(\frac{\cdot}{N}\right) \right) \right)(Nx) ~. 
\end{equation*}
From scaling and \eqref{prob:eq_refined_strichartz}, it then follows that 
\begin{equation*}
\| |x|^\alpha \Ht A_{a,\delta} f \|_{L_t^qL_x^p} \lesssim N^{\frac{d}{2}-\alpha - \frac{1}{q} - \frac{d}{p}} 
\delta^{\frac{1}{2}- \frac{1}{\min(p,q)}} \| f \|_{L_x^2}~. 
\end{equation*}
Finally, replacing \( f \) by \( A_{a,\delta}f \) above, we arrive at \eqref{prob:eq_refined_strichartz_scaled}. 
\end{proof}

\begin{lem}[Probabilistic Strichartz Estimates]\label{prob:lem_prob_strichartz}
Let \( f \in \hsrad(\rd)\) with \begin{equation}\label{prob:eq_prob_strichartz_regularity}
 s \geq \frac{d}{2}- \frac{1}{q} - \frac{d}{p} - \alpha- \frac{1}{\gamma} \left( \frac{1}{2} - \frac{1}{\min(p,q)} \right)~,
 \end{equation}
 where \( \gamma \) is as in Definition \ref{in:def_annular_multiplier}. 
Let \( \alpha \) and \( 2 \leq p,q < \infty \) satisfy \eqref{prob:eq_refined_strichartz_cond_1}. Then, we have for all \( 1\leq \sigma < \infty \) that
\begin{equation}\label{prob:eq_prob_strichartz}
\| |x|^\alpha \Ht f^\omega \|_{\lwtx{\sigma}{q}{p}} \lesssim_{p,q,\alpha,s} \sqrt{\sigma} \| f \|_{\hs}~. 
\end{equation}
\end{lem}
\begin{proof}
We prove \eqref{prob:eq_prob_strichartz} only for \( \sigma \geq \max(p,q) \). The general case then follows by Hölder in the \( \omega\)-variable. 
From the square-function estimate (Lemma \ref{prelim:lem_square_function}), Minkowski's integral inequality, Khintchine's inequality, and Corollary \ref{prob:cor_refined_strichartz_scaled}, it follows that 
\begin{align*}
&\| |x|^\alpha \Ht f^\omega \|_{\lwtx{\sigma}{q}{p}}\\
&\| |x|^\alpha \Ht f^\omega_N\|_{L_\omega^\sigma L_t^q L_x^p \ell_N^2}\\
&\leq \| |x|^\alpha \Ht f^\omega_N \|_{\ell_N^2 \ltxw{q}{p}{\sigma}} \\
&\lesssim \sqrt{\sigma} \| |x|^\alpha \Ht A_kf_N \|_{\ell_N^2\ltxkN{q}{p}{2}} \\
&\leq \sqrt{\sigma} \| |x|^{\alpha} \Ht A_k f_N \|_{\ell_N^2\lkNtx{2}{q}{p}} \\
&\leq \sqrt{\sigma} \| N^{\frac{d}{2}-\alpha - \frac{1}{q} - \frac{d}{p}} 
~ \left( N^{-\frac{1}{\gamma}} \right)^{\frac{1}{2}- \frac{1}{\min(p,q)}} A_k f_N \|_{ \ell_N^2 \ell_k^2 L_x^2}\\
&\leq \sqrt{\sigma} \| N^s f_N \|_{\ell_N^2 L_x^2} \\
&\leq \sqrt{\sigma} \| f \|_{H_x^s}~. 
\end{align*}
We remark that \( f_1 \) is only localized to frequencies \( \lesssim 1 \), so that the inhomogeneous Sobolev norm above is necessary. 
\end{proof}

\begin{lem}[Probabilistic \( L_x^\infty \)-Strichartz Estimates]\label{prob:lem_prob_strichartz_p_infty}
Let \( f_N \in L_{\text{rad}}^2(\rthree) \) and let \( f^\omega_N \) be its radial randomization. Then, we have that 
\begin{align*}
\| |x|^{\frac{3}{8}} \Ht f_N^\omega \|_{L_\omega^\sigma L_t^{\frac{8}{3}} L_x^\infty} &\lesssim \sqrt{\sigma} N^{\frac{3}{4}-\frac{1}{8\gamma}} \| f_N \|_{L_x^2}~,\\
\| |x|^{\frac{1}{4}} \Ht f_N^\omega \|_{L_\omega^\sigma L_t^{4} L_x^\infty} &\lesssim \sqrt{\sigma} N^{1-\frac{1}{4\gamma}} \| f_N \|_{L_x^2}~,\\
\end{align*}
\end{lem}
\begin{rem}
Since \( p = \infty \), we can no longer use the usual combination of Minkowski's integral inequality and Khintchine's inequality. We resolve this by using a radial Sobolev embedding. 
\end{rem}
\begin{proof}
Let \( 1\leq p < \infty \) be a sufficiently large exponent. Using Proposition \ref{prelim:prop_radial_sobolev} and Lemma \ref{prob:lem_prob_strichartz} , we have for all \( p \leq \sigma < \infty \) that 
\begin{equation*}
\| |x|^{\frac{3}{8}} \Ht f_N^\omega \|_{L_\omega^\sigma L_t^{\frac{8}{3}} L_x^\infty}  
\lesssim \| |x|^{\frac{3}{8}} \Ht |\nabla|^{\frac{3}{p}} f_N^\omega\|_{L_\omega^\sigma L_t^{\frac{8}{3}} L_x^p}
\lesssim \sqrt{\sigma} N^{\frac{3}{4}-\frac{1}{8\gamma}} \| f_N \|_{L_x^2}. 
\end{equation*}
Note that, due to scaling, the parameter \( p \) does not appear in the final estimate. Similarly, we have that 
\begin{equation*}
\| |x|^{\frac{1}{4}} \Ht f_N^\omega \|_{L_\omega^\sigma L_t^{4} L_x^\infty} 
\lesssim \| |x|^{\frac{1}{4}} \Ht |\nabla|^{\frac{3}{p}} f_N^\omega \|_{L_\omega^\sigma L_t^{4} L_x^p} 
\lesssim \sqrt{\sigma} N^{1-\frac{1}{4\gamma}} \| f_N \|_{L_x^2}~.
\end{equation*}
\end{proof}

\begin{lem}[Probabilistic \(L_t^\infty\)-Strichartz Estimates]\label{prob:lem_prob_strichartz_q_infty}
Let \( f\in L^2_{\text{rad}}(\rthree) \) and let \( \delta > 0 \).
Then, we have for all \( 1 \leq \sigma < \infty \) and all \( N \in 2^{\mathbb{Z}} \) that 
\begin{align}\label{prob:eq_prob_strichartz_infty}
\| \Ht f^\omega_N \|_{\lwtx{\sigma}{\infty}{6}} &\lesssim \sqrt{\sigma}  N^{1-\frac{1}{3\gamma}} \| f_N \|_{L_x^2} ~, \\
\| |x|^{\frac{1}{2}} \Ht f^\omega_N \|_{\lwtx{\sigma}{\infty}{\infty}} &\lesssim_\delta \sqrt{\sigma}  N^{1-\frac{1-\delta}{2\gamma}} \| f_N \|_{L_x^2} ~.
\end{align}
\end{lem}
\begin{rem}
Since \( q= \infty \), we can no longer use the same combination of Minkowski's integral inequality and Khintchine's inequality as in the proof of Lemma \ref{prob:lem_prob_strichartz}. The same problem was encountered in previous works using the Wiener randomization. In \cite[Proposition 3.3]{OP16}, a chaining-type method was used to bound  \( L_t^\infty\)-norms on compact time intervals. In \cite[Proposition 2.10]{KMV17}, the authors obtain global control on an \( L_t^\infty\)-norm via the fundamental theorem of calculus. Here we present a slight modification of their argument. An alternative approach consists of using a fractional Sobolev embedding in time \cite{DLM18}. 
\end{rem}

\begin{proof} Let \( 1 < q < \infty \) be sufficiently large and assume that \( \sigma \geq q \). We fix \( t_0,t_1\in \mathbb{R} \). By the fundamental theorem of calculus, it holds that 
\begin{align*}
\| \exp(it_1|\nabla|) f_N^\omega \|_{L_x^6} &\leq \| \exp(it_0|\nabla|) f_N^\omega \|_{L_x^6} + \int_{[t_0,t_1]} \| \partial_t ( \exp(it|\nabla|) f_N^\omega) \|_{L_x^6} \dt \\
											&\lesssim\| \exp(it_0|\nabla|) f_N^\omega \|_{L_x^6} +  N \int_{[t_0,t_1]} \| \exp(it|\nabla|) f_N^\omega \|_{L_x^6} \dt \\
											&\lesssim\| \exp(it_0|\nabla|) f_N^\omega \|_{L_x^6} +  N (t_1-t_0)^{\frac{1}{q^\prime}} \| \exp(it|\nabla|) f_N^\omega \|_{L_t^qL_x^6(\mathbb{R}\times \rthree)} ~.
\end{align*}
By taking the \(q\)-th power of this inequality and integrating over \( t_0 \in [t_1-N^{-1},t_1+N^{-1}]\), we obtain that 
\begin{equation*}
\| \exp(it_1|\nabla|) f_N^\omega \|_{L_x^6}^q  \lesssim N  \| \exp(it|\nabla|) f_N^\omega \|_{L_t^qL_x^6(\mathbb{R}\times \rthree)}^q ~. 
\end{equation*}
Taking the supremum in \( t_1 \) and using Lemma \ref{prob:lem_prob_strichartz}, it follows that 
\begin{equation*}
\| \Ht f^\omega_N \|_{\lwtx{\sigma}{\infty}{6}}  \lesssim N^{\frac{1}{q}}  \| \exp(it|\nabla|) f_N^\omega \|_{L_\omega^\sigma L_t^qL_x^6(\mathbb{R}\times \rthree)}
 \lesssim \sqrt{\sigma} N^{1-\frac{1}{3\gamma}} \| f_N \|_{L_x^2}~. 
\end{equation*}
Using the radial Sobolev embedding (Prop. \ref{prelim:prop_radial_sobolev}), Proposition  \ref{prelim:prop_mikhlin}, and the same argument as before, we obtain that 
\begin{align*}
\| |x|^{\frac{1}{2}} \Ht f^\omega_N \|_{\lwtx{\sigma}{\infty}{\infty}} &\lesssim \| |x|^{\frac{1}{2}} \Ht |\nabla|^{\frac{3}{q}}f^\omega_N \|_{\lwtx{\sigma}{\infty}{q}} \\
							&\lesssim N^{\frac{1}{q}+ \frac{3}{q}}  \| |x|^{\frac{1}{2}} \Ht f^\omega_N \|_{\lwtx{\sigma}{q}{q}} \\
							&\lesssim \sqrt{\sigma} N^{1- \frac{1}{\gamma} \left(\frac{1}{2}-\frac{1}{q} \right)} \| f_N \|_{L_x^2}~. 
\end{align*}
This completes the proof of the second estimate.
\end{proof}

\section{An in/out decomposition}\label{sec:decomp}
\begin{figure}[t!]
\begin{center}
\begin{tikzpicture}[scale=2]
%Coordinate Axis
\draw[very thick,<->] (0,2)-- (0,0) -- (2.5,0) ;
\node[below right] at (2.5,0) {\large $r$};
\node[above left] at  (0,2) {\large $t$};
%Light cone
\draw[thick] (0,0) -- (1.9,1.9);
\node[above] at (1.9,1.9) {{ $t=r$} };
%Incoming
\draw [thick, blue, decoration={markings,mark=at position 0.7 with
    {\arrow[scale=2,>=stealth]{>}}},postaction={decorate}] (1,0) -- (0,1);
\node at (0.65,0.15) {{\textcolor{blue}{in}}};
\draw [thick, blue, decoration={markings,mark=at position 2/3*0.7+1/3 with
    {\arrow[scale=2,>=stealth]{>}}},postaction={decorate}] (2,0) -- (0.5,1.5);
\node at (0.65,0.16) {{\textcolor{blue}{in}}};
\node at (1.7,0.16) {{\textcolor{blue}{in}}};
%Outgoing
\draw [thick, red, decoration={markings,mark=at position 0.4 with
    {\arrow[scale=2,>=stealth]{>}}},postaction={decorate}] (0,1) -- (0.5,1.5);
\draw [thick, red, decoration={markings,mark=at position 0.5 with
    {\arrow[scale=2,>=stealth]{>}}},postaction={decorate}] (1,0) -- (2.5,1.5);
    \draw [thick, red, decoration={markings,mark=at position 0.5 with
    {\arrow[scale=2,>=stealth]{>}}},postaction={decorate}] (2,0) -- (2.5,0.5);
\node at (0.4,1.15) {{\textcolor{red}{out}}};
\node at (1.4,0.15) {{\textcolor{red}{out}}};
\node at (2.4,0.15) {{\textcolor{red}{out}}};
%point 
\draw[thick, fill=black] (1/2,3/2) circle[radius=0.2ex];
\end{tikzpicture}
\end{center}
\caption*{\small{We display the in/out-decomposition for radial solutions of the linear wave equation in  \( d=3 \). The blue lines correspond to incoming waves and the red lines correspond to outgoing waves. The incoming wave will be reflected at the origin and transformed into an outgoing wave.}}
\caption{In/out-decomposition}
\label{fig:in/out}
\end{figure}

In this section, we describe a decomposition of solutions to the linear wave equation into incoming and outgoing components (see Figure \ref{fig:in/out}). This decomposition relies heavily on the spherical symmetry of the initial data. The in/out-decomposition can be derived in physical space by using spherical means, see e.g. \cite{Sogge08}. However, for our purposes it is more convenient to derive the decomposition in frequency space. A similar method has been used for the mass-critical NLS in \cite{KTV08}. 

Let \( f \in L^2_{\text{rad}}(\rthree) \) be spherically symmetric. Using the explicit expression \( J_{\frac{1}{2}}(x) = \sqrt{\frac{2}{\pi x}} \sin(x) \) (cf. \cite{bell2004}), it follows that 

\begin{align*}
&\cos(t|\nabla|) f(r)\\
 &=  r^{-\frac{1}{2}}  \int_{0}^{\infty} \cos(t\rho) J_{\frac{1}{2}}(r\rho) \hat{f}(\rho) \rho^{\frac{3}{2}} \drho \\
 &=\sqrt{\frac{2}{\pi}} ~ \frac{1}{r} \int_0^\infty  \cos(t\rho) \sin(r\rho) \hat{f}(\rho) \rho \drho \\
 &= \frac{1}{\sqrt{2\pi}} ~ \frac{1}{r} \int_0^{\infty} (\sin((t+r)\rho)-\sin((t-r)\rho)) \hat{f}(\rho) \rho \drho ~.
\end{align*}
By defining \begin{equation}
W_s[h](\tau) =  \frac{1}{\sqrt{2\pi}} \int_{0}^{\infty} \sin(\tau \rho) h(\rho) \rho \drho ~,
\end{equation}
it follows that \begin{equation*}
\cos(t|\nabla|) f = \frac{1}{r} ( W_s[\hat f](t+r) - W_s[\hat f](t-r) ) ~. 
\end{equation*}
Next, let us derive the corresponding decomposition for the operator \( \sin(t|\nabla|)/|\nabla| \). Let \( g \in \dot{H}_x^{-1}(\rthree) \) be spherically symmetric. Then, 
\begin{align*}
\frac{\sin(t|\nabla|)}{|\nabla|} g(r) &= r^{-\frac{1}{2}} \int_0^\infty \sin(t\rho) J_{\frac{1}{2}}(r\rho) \widehat{g}(\rho) \rho^{\frac{1}{2}} \drho \\
&= \sqrt{\frac{2}{\pi}} \frac{1}{r} \int_0^\infty \sin(t\rho) \sin(r\rho) \widehat{g}(\rho) \drho \\
&= \frac{1}{\sqrt{2\pi}} \frac{1}{r}\int_0^\infty (\cos((t-r)\rho)-\cos((t+r)\rho)) \widehat{g}(\rho) \drho \\
\end{align*}
By defining
\begin{equation*}
W_c[h](\tau) = \frac{1}{\sqrt{2\pi}} \int_0^{\infty} \cos(\tau \rho) h(\rho) \rho \drho ~, 
\end{equation*}
it follows that
\begin{equation*}
\frac{\sin(t|\nabla|)}{|\nabla|} g = r^{-1} \left( -W_c[\rho^{-1} \hat g](t+r) + W_c[\rho^{-1} \hat g](t-r) \right) ~. 
\end{equation*}
Thus, the solution \( F \) of the linear wave equation with initial data \( (f,g) \in L_{\text{rad}}^2(\rthree)\times \dot{H}_{\text{rad}}^{-1}(\rthree) \) is given by 
\begin{equation*}
F(t,x) = \frac{1}{r} \left( W_s[\widehat{f}](t+r)	- W_c[\rho^{-1}\widehat{g}](t+r) -W_s[\widehat{f}](t-r)+ W_c[\rho^{-1}\widehat{g}](t-r) \right)
\end{equation*}
\begin{definition}[In/out-decomposition]\label{decomp:lem_in_out}
Let  \( (f,g) \in L_{\text{rad}}^2(\rthree)\times \dot{H}_{\text{rad}}^{-1}(\rthree) \) and let \( F \) be the corresponding solution to the linear wave equation. 
Then, we define
\begin{align*}
W_{\text{in}}[F](\tau) &= W_s[\widehat{f}](\tau) - W_c[\rho^{-1} \widehat{g}](\tau)~,\\
W_{\text{out}}[F](\tau)& = -W_s[\widehat{f}](\tau) + W_c[\rho^{-1} \widehat{g}](\tau)~.
\end{align*}
As a consequence, we have that
\begin{equation}\label{decomp:eq_in_out}
F(t,x) = \frac{1}{r} \left(W_{\text{in}}[F](t+r)+ W_{\text{out}}[F](t-r)  \right)~.
\end{equation}
\end{definition}
Even though \( W_{\text{in}}[F] \) equals \( -W_{\text{out}}[F] \) we introduced to different notations to serve as a visual aid. This also allows us to savely leave out the arguments \( t+r \) and \( t-r \) in subsequent computations. 

From Plancherell's theorem, it follows that 
\begin{equation}\label{decomp:eq_plancherell}
\| W_s[h](\tau)\|_{L_\tau^2(\mathbb{R})} + \| W_c[h] \|_{L_\tau^2(\mathbb{R})} \lesssim \| \rho h \|_{L_\rho^2(\mathbb{R}_{>0})}~.
\end{equation}
As a consequence, we have that 
\begin{equation}\label{decomp:eq_plancherell}
\| W_{\text{in}}[F](\tau) \|_{L_\tau^2(\mathbb{R})}  + \| W_{\text{out}}[F](\tau) \|_{L_\tau^2(\mathbb{R})} \lesssim \| f \|_{L_x^2(\rthree)} + \| g \|_{\dot{H}^{-1}_x(\rthree)} ~. 
\end{equation}

In the analysis of the Morawetz error term (see Section \ref{sec:morawetz_error}), we will need to control an interaction between \( \nabla F \) and the nonlinear part \( v \). However, the individual components of \( \nabla F \)  are not radial. To overcome this technical problem, we write 
\begin{align*}
&\partial_{x_j} F(t,x) \\
&= \frac{x_j}{r} \partial_r F(t,r) \\
&= - \frac{x_j}{r^3} \left( W_{\text{out}}[F](t-r) + W_{\text{in}}[F](t+r) \right) + \frac{x_j}{r^2} \left( -(\partial_\tau W_{\text{out}}[F])(t-r) + (\partial_\tau W_{\text{in}}[F])(t+r) \right) 
\end{align*}
After a short calculation, we see that
\begin{equation*}
\partial_\tau W_s[\hat{f}](\tau) = W_c[\rho \hat{f}](\tau) \quad \text{and} \quad \partial_\tau W_c[\rho^{-1} \hat{g}](\tau) = - W_s[\hat{g}](\tau)~. 
\end{equation*}
Then, we define
\begin{align}
 \wingradf(\tau) &:=  W_c[\rho \hat{f}](\tau) + W_s[\hat{g}](\tau) ~, \label{decomp:eq_grad_in}\\
\woutgradf(\tau)  &:= W_c[\rho \hat{f}](\tau) + W_s[\hat{g}](\tau) ~.\label{decomp:eq_grad_out}
\end{align}
Using these definitions, it follows that
\begin{equation}\label{decomp:eq_grad_decomp}
\partial_{x_j} F(t,x) = - \frac{x_j}{r^2} F(t,x) + \frac{x_j}{r^2} \left(\woutgradf(t-r)+ \wingradf(t+r) \right)~. 
\end{equation}
Using the same argument as above, we have that
\begin{equation*}
\| \woutgradf (\tau) \|_{L_\tau^2(\mathbb{R})} + \| \wingradf (\tau) \|_{L_\tau^2(\mathbb{R})} \lesssim \| f \|_{\dot{H}_x^1(\rthree)} + \| g \|_{L_x^2(\rthree)}~. 
\end{equation*}

\begin{lem}\label{decomp:lem_w_s_estimate}
Let \( f \in L^2_{\text{rad}}(\mathbb{R}^3)\) be such that \( \supp (\hat{f}) \subseteq \{ \xi\colon |\xi| \in [a,(1+\delta)a] \} \). Then, we have for all \( 2\leq q \leq \infty \) that 
\begin{equation}\label{decomp:eq_improved_integrability}
\| W_s[f](\tau) \|_{L^q_\tau(\mathbb{R})} + \| W_c[f](\tau) \|_{L^q_\tau(\mathbb{R})} \lesssim (a\delta)^{\frac{1}{2}-\frac{1}{q}} \| f \|_{L_x^2(\mathbb{R}^3)} ~.
\end{equation}
\end{lem}
\begin{proof} Using Hölder's inequality, we have that
\begin{equation*}
|W_s[f](\tau)| + |W_c[f](\tau)| \lesssim \int_{a}^{(1+\delta)a} | \hat{f}(\rho)| \rho \drho\leq  (a\delta)^{\frac{1}{2}} 
\left( \int_{0}^\infty |\hat{f}(\rho)|^2 \rho^2 \drho \right)^{\frac{1}{2}} = (a\delta)^{\frac{1}{2}} \| f \|_{L_x^2(\mathbb{R}^3)}~. 
\end{equation*}
This proves \eqref{decomp:eq_improved_integrability} for \( q=\infty \). Together with
 \eqref{decomp:eq_plancherell}, the general case follows by interpolation. 
\end{proof}
Lemma \ref{decomp:lem_w_s_estimate} is the analog of the square-function estimate \cite[Lemma 2.2]{DLM17} for the Wiener randomization. However, since \( f \) is radial, it is much easier to prove. 

\begin{cor}[Improved integrability for the in/out decomposition]\label{decomp:cor_integrability}
Let \( f \in L_{\text{rad}}^2(\mathbb{R}^3) \). Then, we have for all \(  2 \leq q < \infty \) that 
\begin{equation*}
\| W_s[f_N^\omega](\tau)\|_{L_\omega^\sigma L_\tau^q} + \| W_c[f_N^\omega](\tau)\|_{L_\omega^\sigma L_\tau^q} \lesssim N^{(1-\frac{1}{\gamma}) (\frac{1}{2}-\frac{1}{q})} \| f_N \|_{L_x^2(\rthree)} ~.
\end{equation*}
\end{cor}
\begin{proof}
As in Section \ref{section:prob_strichartz}, we restrict to the case \( q \leq \sigma < \infty \). Using a combination of Khintchine's inequality, Minkowski's integral inequality, and Lemma \ref{decomp:lem_w_s_estimate}, we have that
\begin{align*}
&\| W_s[f_N^\omega](\tau)\|_{L_\omega^\sigma L_\tau^q} \\
&\leq \| W_s[f_N^\omega](\tau)\|_{L_\tau^q L_\omega^\sigma } \\
&\lesssim \sqrt{\sigma} \| W_s[A_k f_N ](\tau) \|_{L_\tau^q \ell_k^2 } \\
&\leq \sqrt{\sigma} \| W_s[A_k f_N ](\tau) \|_{ \ell_k^2L_\tau^q } \\
&\lesssim N^{(1-\frac{1}{\gamma}) (\frac{1}{2}-\frac{1}{q})} \| A_k f_N \|_{\ell_k^2 L_x^2} \\
&\lesssim N^{(1-\frac{1}{\gamma}) (\frac{1}{2}-\frac{1}{q})} \| f_N \|_{L_x^2}~.
\end{align*}
The same argument also works for \( W_c[f_N^\omega](\tau) \).
\end{proof}
\begin{rem}
For \( \gamma =1 \), Corollary \ref{decomp:cor_integrability} shows that \(  W_s[f_N^\omega](\tau)\in \bigcap_{2\leq  q < \infty} L_\tau^q(\mathbb{R}) \) almost surely for all \( f \in H_{\text{rad}}^{0+}(\rthree) \). This holds because the radial randomization is similar to a Wiener randomization of the function \( f(r) r \). 
\end{rem}

\section{Local well-posedness and conditional scattering}\label{section:lwp}

Recall that the forced nonlinear wave equation is given by 
\begin{equation}\label{lwp:eq_forced_nlw}
\begin{cases}
-\partial_{tt} v + \Delta v = (v+F)^5~,\qquad (t,x) \in \mathbb{R}\times \mathbb{R}^3~.\\
v(t_0,x)= v_0 \in \dot{H}_x^1(\rthree), \qquad \partial_t v(t_0,x) = v_1\in L_x^2(\rthree)~. 
\end{cases}
\end{equation}
In this section, it is not important that \( F \) solves a linear wave equation. However, this will be essential in Sections \ref{section:energy}-\ref{section:a_priori_bound}.

\begin{lem}[Local Well-Posedness]\label{lwp:lem_lwp}
Let \( (v_0,v_1)\in \dot{H}_x^1(\rthree) \times L_x^2(\rthree) \) and assume that \( F \in L_t^5L_x^{10}(\mathbb{R}\times \rthree) \). Then, there exists a maximal time interval of existence \( I \) and a corresponding unique solution \( v \) of 
\eqref{lwp:eq_forced_nlw} satisfying 
\begin{equation*}
(v,\partial_t v) \in \big( C_t^0 \dot{H}_x^1(\sI) \cap L_{t,\text{loc}}^5L_x^{10}(\sI) \big)\times C_t^0 L_x^2(\sI)~.
\end{equation*}
 Moreover, if both the initial data \( (v_0,v_1) \) and the forcing term \( F \) are radial, 
then \( v \) is also radial. 
\end{lem}
The proof consists of a standard application of Strichartz estimates, and we omit the details. We refer the reader to \cite[Lemma 3.1]{DLM17} and  \cite[Theorem 1.1]{Pocovnicu17} for related results. 
In \cite{Pocovnicu17} the stability theory for energy critical equations  was used to reduce to the proof of almost sure global well-posedness to an \textit{a priori} energy bound. Similar methods have also been used in \cite{BOP2015,DLM17,DLM18,KMV17,OP16}. 

\begin{prop}[{\cite[Theorem 1.3]{DLM17}}]\label{lwp:prop_conditional}\label{lwp:prop_scattering}
Let \( (v_0,v_1) \in \dot{H}_x^1(\rthree) \times L_x^2(\rthree) \) and \( F \in L_t^5L_x^{10}(\mathbb{R}\times \mathbb{R}^3) \). Let \( v(t) \) be a solution 
\eqref{lwp:eq_forced_nlw} and let \( I \) be its maximal time interval of existence. Furthermore, we assume that \( v \) satisfies the \textit{a priori} bound 
\begin{equation}
M:= \sup_{t\in I} E[v](t) < \infty~.
\end{equation}
Then \( v \) is a global solution, it obeys the global space-time bound 
\begin{equation*}
\| v\|_{L_t^5 L_x^{10}(\mathbb{R}\times \rthree)} \leq C(M, \| F \|_{L_t^5L_x^{10}(\mathbb{R}\times \rthree)})< \infty~,
\end{equation*}
and it scatters as \( t \rightarrow \pm \infty \). 
\end{prop}
Theorem 1.3 in \cite{DLM17} is stated for the energy critical NLW in \( d=4 \). However, the same argument also yields Proposition \ref{lwp:prop_conditional}. We point out that the proof crucially relies on the deterministic theory for the energy critical NLW \cite{BG99,Tao06b}. 

\section{Almost energy conservation and decay estimates}\label{section:energy}
In this section, we prove new estimates for the solution to the forced NLW
\begin{equation}\label{flux:eq_forced_nlw}
\begin{cases}
-\partial_{tt} v + \Delta v = (v+F)^5~,\qquad (t,x) \in \mathbb{R}\times \mathbb{R}^3~.\\
v(t_0,x)= v_0 \in \dot{H}_x^1(\rthree), \qquad \partial_t v(t_0,x) = v_1\in L_x^2(\rthree)~. 
\end{cases}
\end{equation}
In contrast to Section \ref{section:lwp}, we now assume that \( F \) is a solution to the linear wave equation. 
Recall that the stress-energy tensor of the energy critical NLW is given by 
\begin{align*}
T^{00} &:= \frac{1}{2} \left( (\partial_t v)^2 + |\nabla v|^2 \right) + \frac{1}{6} v^6~,\\
T^{j0} &:= - \partial_t v ~\partial_{x_j} v \\
T^{jk} &:=  \partial_{x_j} v ~ \partial_{x_k} v - \frac{\delta_{jk}}{4} (-\partial_{tt} + \Delta) ( v^2) +\frac{\delta_{jk}}{3} v^6~.  
\end{align*}
In the above tensor, we have that \( j,k=1,2,3 \). If \( v \) solves the energy critical NLW \eqref{in:eq_nlw}, then the stress-energy tensor is divergence free. This leads to energy conservation, momentum conservation, and several decay estimates, such as Morawetz estimates, flux estimates, or potential energy decay (see \cite{Sogge08,Tao06}). If \( v \) solves the forced nonlinear wave equation \eqref{flux:eq_forced_nlw}, then the stress-energy tensor is no longer divergence free. However, the error terms in the divergence are of lower order, so that we can still hope for almost conservation laws and some decay estimates. More precisely, 
with \( \mathcal{N} := (v+F)^5 - v^5 \), it follows from a standard computation that 
\begin{align}
\partial_t T^{00} + \partial_{x_k} T^{0k} &= - \mathcal{N} \partial_t v   \label{flux:eq_divergence_1}\\
\partial_t T^{j0} + \partial_{x_k} T^{jk} &= \mathcal{N}  \partial_{x_j} v - \frac{1}{2} \partial_{x_j} (\mathcal{N} v ) \label{flux:eq_divergence_2}~. 
\end{align}

For our purposes, the most important quantity measuring the size and regularity of \( v \) is its energy  \begin{equation*}
E[v](t) = \int \frac{1}{2} |\nabla v|^2  + \frac{1}{2} |\partial_t v|^2 + \frac{1}{6} |v|^6 \dx~.
\end{equation*} 
For future use, we also define the local energy as 
\begin{equation*}
e[v](t) := \int_{|x|\leq |t|}\frac{1}{2} |\nabla v|^2  + \frac{1}{2} |\partial_t v|^2 + \frac{1}{6} |v|^6 \dx~.
\end{equation*}
Next, we  determine the error terms in the almost energy conservation law.

\begin{prop}[Energy Increment]\label{flux:prop_energy_increment}
Let \( I = [a,b] \) be a time interval and \( v \colon \sI \rightarrow \mathbb{R} \) be a solution to the forced nonlinear wave equation \eqref{flux:eq_forced_nlw}. Then, we have that 
\begin{align}\label{flux:eq_energy_increment}
&|E[v](b)-E[v](a)| \\
&\lesssim \| F\|_{L_t^\infty L_x^6(\sI)} \sup_{t\in I}E[v](t)^{\frac{5}{6}} + \left| \int_I \int_{\rthree} (\partial_t F) v^5 \dx \dt \right| + \int_I \int_{\rthree} |F|^2 (|F|+|v|)^3 |\partial_t v|\dx \dt ~. \notag
\end{align}
\end{prop}
The first summand on the right-hand side of \eqref{flux:eq_energy_increment} has a lower power in the energy. After placing the random linear evolution in \( L_t^\infty L_x^6(\mathbb{R}\times \rthree) \), it can easily be controlled via a bootstrap argument. The second summand is the main error term in this almost energy conservation law, and we will control it in Section  \ref{section:energy_increment}. Finally, the third summand in \eqref{flux:eq_energy_increment} only includes lower order error terms, and they are controlled in Section \ref{section:minor}. \\
The idea to integrate by parts in the energy increment has previously been used in \cite{DLM18,KMV17,OP16}. 

\begin{proof}
From the divergence formula \eqref{flux:eq_divergence_1}, it follows that 
\begin{align*}
\frac{\mathrm{d}}{\dt} E[v](t) &= \frac{\mathrm{d}}{\dt} \int_{\rthree} T^{00}(t,x) \dx \\
&= - \int_{\rthree} \mathcal{N} \partial_t v \dx \\
&= - 5 \int_{\rthree} F v^4 \partial_t v \dx - \int_{\rthree} \left( 10 F^2 v^3 + 10 F^3 v^2 + 5 F^4 v + F^5 \right) \partial_t v \dx ~. 
\end{align*}
Integrating in time, we obtain that 
\begin{equation}
|E[v](b)-E[v](a)| \lesssim \left|\int_I \int_\mathbb{R}^3 F v^4 \partial_t v \dx \dt \right| + \int_{I}\int_{\mathbb{R}^3} |F|^2 (|F|+|v|)^3 |\partial_t v | \dx \dt ~. \label{flux:eq_energy_increment_proof_1}
\end{equation}
The second summand in \eqref{flux:eq_energy_increment_proof_1} is already acceptable; thus, we now turn to the first summand. Using integration by parts, we have that 
\begin{align*}
 &5 \left|\int_I \int_{\rthree} F v^4 \partial_t v \dx \dt \right| \\
 &=  \left|\int_I \int_{\rthree} F \partial_t(v^5) \dx \dt \right| \\
  &\leq \left|\int_I \int_{\mathbb{R}^3} \partial_t(F)~  v^5 \dx \dt \right| + \int_{\rthree} |F|(b,x) |v|^5(b,x) \dx + \int_{\rthree} |F|(a,x) |v|^5(a,x) \dx \\
  &\lesssim \left|\int_I \int_{\mathbb{R}^3} \partial_t(F)~  v^5 \dx \dt \right| + \| F\|_{L_t^\infty L_x^6(\sI)} \sup_{t\in I} E[v](t)^{\frac{5}{6}}~. 
\end{align*}
Thus, the contribution of the first summand in \eqref{flux:eq_energy_increment_proof_1} is also acceptable. 
\end{proof}
By contracting the stress-energy tensor against different vector fields, one sees that solutions to the energy critical NLW obey a range of decay estimates. One of the most important decay estimates in the study of dispersive equations is the Morawetz estimate, and it has been used to prove almost sure scattering in  \cite{DLM17,DLM18,KMV17}. For the reader's convenience, we recall a classical Morawetz identity. 

\begin{lem}[Morawetz identity]\label{flux:lem_morawetz_identity}
Let \( I = [a,b] \) be a given time interval, and let \( v\colon I \times \rthree \rightarrow \mathbb{R} \) be a solution of \eqref{flux:eq_forced_nlw}. Then, we have the Morawetz identity
\begin{align}
&\tfrac{2}{3} \int_{I} \int_{\mathbb{R}^3} \frac{v^6}{|x|} \dx\dt + \pi \int_I |v|^2(t,0) \dt + 
\int_I \int_{\mathbb{R}^3} | \nabla_{\text{ang}}v|^2 \dx \dt \label{flux:eq_morawetz_identity} \\
&= \int_{\mathbb{R}^3} \partial_t v ~ \frac{x}{|x|}\cdot \nabla v  - 4 \frac{v}{|x|} \partial_t v \dx~ {\Big|_{t=a}^{b}} - \int_I \int_{\mathbb{R}^3} \mathcal{N}~  \frac{x}{|x|} \cdot \nabla v \dx\dt - \int_I \int_{\mathbb{R}^3} \frac{1}{|x|} \mathcal{N} v \dx \dt ~. \notag
\end{align}
Here, \( \nabla_{\text{ang}} v := \nabla v - \frac{x}{|x|} \cdot \nabla v \) denotes the angular component of the gradient of \( v \). 
\end{lem}
The lemma follows along a line of standard computations using \eqref{flux:eq_divergence_1} and \eqref{flux:eq_divergence_2}, see e.g. \cite{Tao06}. We now rewrite the error terms in \eqref{flux:eq_morawetz_identity} more explicitly in terms of \(  F \), and group similar terms together. 
\begin{prop}[Morawetz Estimate]\label{flux:lem_morawetz_estimate}
Let \( I = [a,b] \) be a given time interval, and let \( v\colon I \times \rthree \rightarrow \mathbb{R} \) be a solution of \eqref{flux:eq_forced_nlw}. Then, we have the Morawetz estimate
\begin{align}
&\int_I \int_{\rthree} \frac{v^6}{|x|} \dx \dt \notag\\
&\lesssim \sup_{t\in I} E[v](t) + \left|\int_I \int_{\rthree} \frac{x}{|x|}\cdot \nabla_x(F) ~ v^5 \dx \dt \right|\label{flux:eq_morawetz_major_error} \\
 &~~+ \int_I \int_{\rthree} \frac{1}{|x|} |F| (|v|^5+|F|^5) \dx \dt + \int_I \int_{\rthree} |F|^2 ( |F|+ |v|)^3 \left( \frac{|v|}{|x|} + |\nabla v | \right)~.  \label{flux:eq_morawetz_minor_error} 
\end{align}
\end{prop}
~\\ The second summand in \eqref{flux:eq_morawetz_major_error} is the main error term in this estimate, and we will control it in Section \ref{sec:morawetz_error}. In contrast, the error terms in \eqref{flux:eq_morawetz_minor_error} are easier
to control, and they will be handled in Section \ref{section:minor}. 

\begin{proof}
To prove the proposition, we have to control the terms on the right-hand side of \eqref{flux:eq_morawetz_identity}. First, using Hardy's inequality, we have that
\begin{align*}
&\left| \int_{\mathbb{R}^3} \partial_t v ~ \frac{x}{|x|}\cdot \nabla v  - 4 \frac{v}{|x|} \partial_t v \dx~ {\Big|_{t=a}^{b}} \right| \\
&\lesssim \| \partial_t v(t) \|_{L_t^\infty L_x^2(\sI)}  \| \nabla v \|_{L_t^\infty L_x^2(\sI)} +  \| \partial_t v(t) \|_{L_t^\infty L_x^2(\sI)}  \| \frac{v}{|x|} \|_{L_t^\infty L_x^2(\sI)} \\
&\lesssim \sup_{t\in I} E[v](t)~. 
\end{align*}
Thus, the contribution is acceptable. Second, we have that 
\begin{align*}
&\left| \int_I \int_{\rthree} \mathcal{N} \frac{x}{|x|} \cdot \nabla v\dx \dt \right| \\
&\lesssim \left| \int_I \int_{\rthree} F v^4 \frac{x}{|x|} \cdot \nabla  v \dx \dt \right| + \int_{I} \int_{\rthree} |F|^2 \left( |F| + |v| \right)^3 |\nabla v| \dx \dt \\
&\lesssim  \left| \int_I \int_{\rthree} F  \frac{x}{|x|} \cdot \nabla( v^5)  \dx \dt \right| + \int_{I} \int_{\rthree} |F|^2 \left( |F| + |v| \right)^3 |\nabla v| \dx \dt \\
&\lesssim  \left| \int_I \int_{\rthree} \nabla\cdot \left(F  \frac{x}{|x|}\right)  v^5  \dx \dt \right| + \int_{I} \int_{\rthree} |F|^2 \left( |F| + |v| \right)^3 |\nabla v| \dx \dt \\
&\lesssim  \left| \int_I \int_{\rthree}   \frac{x}{|x|}\cdot \nabla(F)~  v^5  \dx \dt \right| + \int_I \int_{\rthree} \frac{|F|}{|x|} |v|^5 \dx \dt + \int_{I} \int_{\rthree} |F|^2 \left( |F| + |v| \right)^3 |\nabla v| \dx \dt \\
\end{align*}
Thus, the contribution is acceptable. Finally, we have that 
\begin{equation*}
\left| \int_I \int_{\mathbb{R}^3} \frac{1}{|x|} \mathcal{N} v \dx \dt \right| 
\lesssim \int_I \int_{\rthree} \frac{1}{|x|} |F| ( |F|+|v| )^4 |v| \dx \dt 
\lesssim \int_I \int_{\rthree} \frac{1}{|x|} |F| (|F|+|v|)^5 \dx \dt~. 
\end{equation*}
\end{proof}

In contrast to the case \( d=4 \) as in \cite{DLM17,DLM18}, the energy and the Morawetz term are not strong enough to control the main error terms. In addition, we will rely on the following flux estimates on light cones. 
\begin{figure}[t!]
\begin{center}
\begin{tikzpicture}[scale=3]
%Coordinate Axis
\draw[very thick,->] (-2,0) -- (2,0);
\draw[very thick, ->] (0,0) -- (0,2);
\node[below] at (2,0) {\Large $x$};
\node[left] at  (0,2) {\Large $t$};
%Light cone
\draw[thick] (0,0) -- (1.9,1.9);
\draw[thick] (0,0) -- (-1.9,1.9);
\draw[ultra thick, blue] (0.5,0.5) -- (1.5,1.5);
\draw[ultra thick, blue] (-0.5,0.5) -- (-1.5,1.5);
\node[below right] at (1.0,1.0) {{\large \textcolor{blue}{$\leftarrow$ flux}}};
\node[above] at (1.9,1.9) {{\large $t=|x|$} };

%Interval I
\draw[dashed, thick] (-2,1.5) -- (-1.5,1.5);
\draw[dashed, thick] (2,1.5) -- (1.5,1.5);
\draw[very thick,red ] (-1.5,1.5) -- (1.5,1.5);
\draw[dashed, thick] (-2,0.5) -- (-0.5,0.5);
\draw[dashed, thick] (2,0.5) -- (0.5,0.5);
\draw[very thick,red ] (-0.5,0.5) -- (0.5,0.5);
\node[below right] at (0,0.5) {a};
\node[align=left] at (-0.3,0.7) { \large\textcolor{red}{$e(a)$} \\{\textcolor{red}{$ \downarrow $}}};
\node[above right] at (0,1.5) {b};
\node[align=left] at (-0.3,1.3) { {\textcolor{red}{$ \uparrow $}}\\\large \textcolor{red}{$e(b)$}};
\node[align=center] at (-1.9, 1) {\large $I$};
\end{tikzpicture}
\end{center}
\caption*{\small{This figure displays the quantities involved in the forward flux estimate. The local energy at times \( t=a,b\) is the integral of the energy density over the red regions. The flux is the integral of \( v^6 \) over the blue region in space-time. Using the stress-energy tensor, we can control the flux by the increment of the local energy. }}
\caption{Forward Flux Estimate}
\label{fig:forward_flux}
\end{figure}

\begin{lem}[Forward Flux Estimate]\label{flux:lem_basic_flux}
Let \( v \) be a solution of \eqref{flux:eq_forced_nlw} on a  compact time interval \( I=[a,b] \subseteq [0,\infty ) \).  Then, we have that 
\begin{equation}\label{flux:eq_basic_flux}
\tfrac{1}{6} \int_{|x|=t, t\in I} v^6(t,x) \dsigma 
 \leq e[v](b)-e[v](a) + \int_{|x|\leq t, t\in I} \partial_t v \left( (v+F)^5 - v^5 \right) \dx \dt ~.
\end{equation}
\end{lem}
\begin{rem}
The flux estimate is a monotonicity formula based on the increment of the local energy. The term on the left-hand side of \eqref{flux:eq_basic_flux} describes the inflow of potential energy through the light cone. 
\end{rem}
\begin{proof}
We have that 
\begin{align*}
&~~\frac{\mathrm{d}}{\mathrm{d}t} e[v](t) \\
&= \int_{|x|=t} \tfrac{1}{2} |\nabla v|^2  + \tfrac{1}{2} |\partial_t v|^2 + \tfrac{1}{6} |v|^6  \dsigma 
+ \int_{|x|\leq t} \partial_t \nabla v ~ \nabla v + \partial_{tt} v \partial_t v + v^5 \partial_t v  \dx \\
& = \int_{|x|=t} \tfrac{1}{2} |\nabla v|^2  + \tfrac{1}{2} |\partial_t v|^2 - \partial_t v ~\nabla v \cdot \vec{n}+ \tfrac{1}{6} |v|^6  \dsigma \\
&~~~~+ \int_{|x|\leq t} \partial_t v ( \partial_{tt} v - \Delta v + v^5) \dx \\
&\geq \int_{|x|=t} \tfrac{1}{6} |v|^6 \dsigma + \int_{|x|\leq t } \partial_t v (-(v+F)^5+v^5) \dx ~.
\end{align*}
Integrating over \( t \in I \), we arrive at \eqref{flux:eq_basic_flux}. 
\end{proof}
The estimate \eqref{flux:eq_basic_flux} by itself is not useful. Indeed, it only controls the size of \( v \) on a lower-dimensional surface in space-time. We will now use time-translation invariance to integrate it against a weight \(w\in L_\tau^1(\mathbb{R}) \).

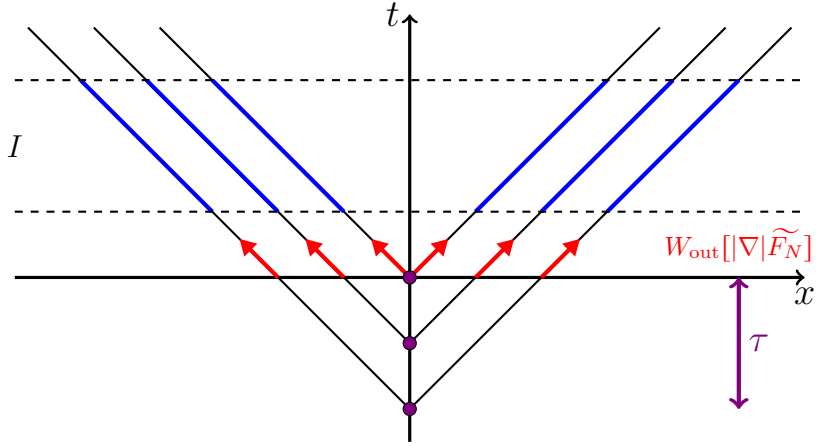
\begin{figure}[!t]
\begin{center}
\begin{tikzpicture}[scale=1.75,
  decoration={markings,mark=at position 1 with {\fill[red] (3pt,0)--(-3pt,3.47pt)--(-3pt,-3.47pt)--cycle;}},
   ]
   %Coordinate Axis
\draw[very thick,->] (-3,0) -- (3,0);
\draw[very thick, ->] (0,-1.25) -- (0,2);
\node[below] at (3,0) {\Large $x$};
\node[left] at  (0,2) {\Large $t$};
%Light cone
\draw[thick] (0,0) -- (1.9,1.9);
\draw[thick] (0,0) -- (-1.9,1.9);
\draw[ultra thick, blue] (0.5,0.5) -- (1.5,1.5);
\draw[ultra thick, blue] (-0.5,0.5) -- (-1.5,1.5);

%Shifted Cone 1
\draw[thick] (0,-0.5) -- (2.4,1.9);
\draw[thick] (0,-0.5) -- (-2.4,1.9);
\draw[ultra thick, blue] (1,0.5) -- (2,1.5);
\draw[ultra thick, blue] (-1,0.5) -- (-2,1.5);

%Shifted Cone 2
\draw[thick] (0,-1) -- (2.9,1.9);
\draw[thick] (0,-1) -- (-2.9,1.9);
\draw[ultra thick, blue] (1.5,0.5) -- (2.5,1.5);
\draw[ultra thick, blue] (-1.5,0.5) -- (-2.5,1.5);

%Interval I
\draw[dashed, thick] (-3,1.5) -- (3,1.5);
\draw[dashed, thick] (-3,0.5) -- (3,0.5);
\node[align=center] at (-3, 1) {\large $I$};

%Red arrows
\draw[ultra thick, red, draw opacity=1, postaction=decorate] (0,0) -- (0.25,0.25);
\draw[ultra thick, red, draw opacity=1, postaction=decorate] (0,0) -- (-0.25,0.25);
\draw[ultra thick, red, draw opacity=1, postaction=decorate] (0.5,0) -- (0.75,0.25);
\draw[ultra thick, red, draw opacity=1, postaction=decorate] (-0.5,0) -- (-0.75,0.25);
\draw[ultra thick, red, draw opacity=1, postaction=decorate] (1,0) -- (1.25,0.25);
\draw[ultra thick, red, draw opacity=1, postaction=decorate] (-1,0) -- (-1.25,0.25);

\node at (2.5,0.25) {\small \textcolor{red}{$W_{\text{out}}[|\nabla|\widetilde{F_N}]$}};

%Shift Line
\draw[ultra thick,violet, <->] (2.5,-1) -- (2.5,0);
\node[right] at (2.5,-0.5) {\textcolor{violet}{\Large $\tau$}};
\draw[ black, fill=violet] (0,0) circle[radius=0.05];
\draw[ black, fill=violet] (0,-0.5) circle[radius=0.05];
\draw[ black, fill=violet] (0,-1) circle[radius=0.05];
\end{tikzpicture}
\end{center}
\caption*{\small{We display the idea behind the interaction flux estimate. By using the time-translation invariance of the equation, we can control \( v^6 \) on the blue region of each shifted light cone. Then, we integrate the forward flux estimate against a weight \( w \) depending only on the shift \( \tau \). Since the outgoing component \(W_{\text{out}}[|\nabla|\widetilde{F_N}](t-|x|) \) is constant on forward light cones, we choose \( w= |W_{\text{out}}[|\nabla|\widetilde{F_N}]|^2 \). }}
\caption{Interaction Flux Estimate}
\label{fig:interaction_flux}
\end{figure}

\begin{prop}[Forward Interaction Flux Estimate]\label{flux:prop_forward_interaction_flux}
Let \( v \) be a solution to the forced NLW  \eqref{flux:eq_forced_nlw} on a compact time interval \( I=[a,b]\subseteq [0,\infty) \). Also, let \( w \in L_\tau^1(\mathbb{R}) \) be nonnegative. Then, we have that 
\begin{align}
&\int_{I} \int_{\rthree} w(t-|x|) |v|^6(t,x) \dx \dt \notag \\
&\lesssim \| w \|_{L_\tau^1(\mathbb{R})} \sup_{t\in I} E[v](t) + \| w \|_{L_\tau^1(\mathbb{R})} \| F \|_{L_t^\infty L_x^6(\sI)} \sup_{t\in I} E[v](t)^{\frac{5}{6}} \label{flux:eq_interaction_flux_error_boundary}\\
&~~+ \left| \int_I \int_{\rthree} \left( \int_{-\infty}^{t-|x|} w(\tau) \dtau\right) \partial_t(F) v^5\dx \dt \right| +\left| \int_I \int_{\rthree} w(t-|x|) F~ v^5\dx \dt \right|  \label{flux:eq_interaction_flux_error_major}\\
&~~+\| w \|_{L_\tau^1(\mathbb{R})} \int_{I} \int_{\mathbb{R}^3} |F|^2( |F|+|v|)^3 |\partial_t v |\dx \dt \label{flux:eq_interaction_flux_error_minor}
\end{align}
\end{prop}
In order to control the energy, we essentially choose \( w \) as the outgoing component of the linear wave \( F \) (cf. Section \ref{section:bootstrap} and Figure \ref{fig:interaction_flux}). \\
The terms in \eqref{flux:eq_interaction_flux_error_boundary} correspond to boundary terms, and they can easily be controlled by a bootstrap argument. The main error terms are in \eqref{flux:eq_interaction_flux_error_major}, 
and they will be controlled in Section \ref{section:interaction_flux}. In contrast, the errors in \eqref{flux:eq_interaction_flux_error_minor} are of lower order, and they will be controlled in Section \ref{section:minor}. \\
To remember that the weight \( w \) in \eqref{flux:eq_interaction_flux_error_major} should be integrated over \( {(-\infty, t-|x|]} \), note that the contribution of the error \( \partial_t(F) v^5 \) should be weighted less as \(t\rightarrow -\infty \) and \( |x|\rightarrow \infty \). 

\begin{proof}
By time-translation invariance and Lemma \ref{flux:lem_basic_flux}, we obtain for any \( \tau \in \mathbb{R} \) that 
\begin{align}
&~~~~\int_{|x|=t-\tau, t\in I} \frac{|v|^6(t,x)}{6} \dsigma \notag \\
&\leq \int_{|x|\leq b - \tau}  \frac{1}{2} |\nabla v|^2  + \frac{1}{2} |\partial_t v|^2 + \frac{1}{6} |v|^6 \dx \Big|_{t=b} 
 - \int_{|x|\leq a - \tau}  \frac{1}{2} |\nabla v|^2  + \frac{1}{2} |\partial_t v|^2 + \frac{1}{6} |v|^6 \dx\Big|_{t=a} \notag \\
 &~~~~+ \int_{|x|\leq t - \tau, t \in I} \partial_t v ((v+F)^5-v^5) \dx \dt  \notag \\
 &\leq 2 \sup_{t\in I} E[v](t)+ \int_{|x|\leq t - \tau, t \in I} \partial_t v ((v+F)^5-v^5) \dx \dt ~. \label{flux:eq_translated_flux}
\end{align}
Integrating \eqref{flux:eq_translated_flux} against the function \( w(\tau) \), we obtain that 
\begin{align}
&~~~~\frac{1}{6} \int_I \int_{\mathbb{R}^3} w(t-|x|) |v(t,x)|^6 \dx \dt\notag \\
&=\frac{1}{6} \int \int_{|x|=t - \tau} w(\tau) |v|^6(t,x) \dsigma \dtau \notag \\
&\leq 2 \| w \|_{L^1_\tau(\mathbb{R})} \sup_{t\in I} E[v](t) 
+ \int \int_{|x|\leq t-\tau, t \in I} w(\tau) \partial_t v(t) ((v+F)^5-v^5) \dx \dt \dtau \notag \\
&\lesssim \| w \|_{L^1_\tau(\mathbb{R})} \sup_{t\in I} E[v](t) + \left| \int_{\mathbb{R}} \int_{|x|\leq t-\tau, t\in I} w(\tau) F v^4 \partial_t v \dx \dt \dtau \right|		\label{flux:eq_flux_proof_1}\\
&~~~+ \| w \|_{L_\tau^1(\mathbb{R})} \int_{I} \int_{\mathbb{R}^3} |F|^2( |F|+|v|)^3 |\partial_t v |\dx \dt~. \notag
\end{align}
The first and third summand in \eqref{flux:eq_flux_proof_1} are acceptable contributions. Thus, we turn to the second summand in \eqref{flux:eq_flux_proof_1}. Using integration by parts, we have that 
\begin{align*}
&5 \left| \int_I \int_{|x|\leq t-\tau, t \in I} w(\tau) F v^4 \partial_t v \dx \dt \dtau \right| \\
&= \left| \int_I \int_{\rthree} \left( \int_{-\infty}^{t-|x|} w(\tau) \dtau\right) F \partial_t(v^5)\dx \dt \right| \\
&=\left| \int_{\rthree}  \left( \int_{-\infty}^{t-|x|} w(\tau) \dtau\right) F v^5 \dx \Big|_{t=b} \right|  + \left| \int_{\rthree}  \left( \int_{-\infty}^{t-|x|} w(\tau) \dtau\right) F v^5 \dx \Big|_{t=a} \right| \\
&~~+ \left| \int_I \int_{\rthree} \left( \int_{-\infty}^{t-|x|} w(\tau) \dtau\right) \partial_t(F) v^5\dx \dt \right|
+\left| \int_I \int_{\rthree} w(t-|x|) F~ v^5\dx \dt \right| \\
&\lesssim \| w \|_{L_\tau^1(\mathbb{R})} \| F \|_{L_t^\infty L_x^6(\sI)} \sup_{t\in I} E[v](t)^{\frac{5}{6}} + \left| \int_I \int_{\rthree} \left( \int_{-\infty}^{t-|x|} w(\tau) \dtau\right) \partial_t(F) v^5\dx \dt \right| \\
&~~+\left| \int_I \int_{\rthree} w(t-|x|) F~ v^5\dx \dt \right| \\
\end{align*}

\end{proof}

By replacing the forward light-cones in the derivation of Proposition \ref{flux:prop_forward_interaction_flux} by backward light-cones, one easily derives the following proposition.
\begin{prop}[Backward Interaction Flux Estimate]\label{flux:prop_backward_interaction_flux}
Let \( v \) be a solution of  \eqref{flux:eq_forced_nlw} on a compact time interval \( I=[a,b]\subseteq [0,\infty) \). Also, let \( w \in L_\tau^1(\mathbb{R}) \) be nonnegative. Then, we have that 
\begin{align}
&\int_{I} \int_{\rthree} w(t+|x|) |v|^6(t,x) \dx \dt \label{flux:eq_backward_interaction_flux} \\
&\lesssim \| w \|_{L_\tau^1(\mathbb{R})} \sup_{t\in I} E[v](t) + \| w \|_{L_\tau^1(\mathbb{R})} \| F \|_{L_t^\infty L_x^6(\sI)} \sup_{t\in I} E[v](t)^{\frac{5}{6}} \notag \\
&~~+ \left| \int_I \int_{\rthree} \left( \int_{t+|x|}^{\infty} w(\tau) \dtau\right) \partial_t(F) v^5\dx \dt \right| +\left| \int_I \int_{\rthree} w(t+|x|) F~ v^5\dx \dt \right|  \label{flux:eq_backward_flux_major} \\
&~~+\| w \|_{L_\tau^1(\mathbb{R})} \int_{I} \int_{\mathbb{R}^3} |F|^2( |F|+|v|)^3 |\partial_t v |\dx \dt \notag
\end{align}
\end{prop}
To remember that the weight \( w \) in \eqref{flux:eq_backward_flux_major} should be integrated over \( {[t+|x|,\infty )} \), note that the contribution of the error \( \partial_t(F) v^5 \) should be weighted less as \( t,|x|\rightarrow \infty \).

\section{Bootstrap argument}\label{section:bootstrap}
In this section, we introduce the quantities in the bootstrap argument to control the energy. For a given time interval \( I\subseteq \mathbb{R} \), we define the energy 
\begin{equation}\label{bootstrap:eq_energy}
\mathcal{E}_I := \sup_{t\in I} E[v](t) = \sup_{t\in I} \int_{\mathbb{R}^3} \frac{1}{2} (\partial_t v(t,x))^2 +  \frac{1}{2} |\nabla v(t,x)|^2 + \frac{1}{6} |v(t,x)|^6 \dx~
\end{equation}
and the Morawetz term 
\begin{equation}\label{bootstrap:eq_morawetz}
\mathcal{A}_I := \| |x|^{-\frac{1}{6}} v \|_{L_{t,x}^6(I \times \mathbb{R}^3)}^6~.
\end{equation}
Before we can define the interaction flux term, we need to introduce some further notation. 
Let \( F \) be a solution to the linear wave equation with initial data \( F|_{t=0} = f_0\in L_{\rad}^2(\rthree) \) and \( \partial_t F|_{t=0} = g_0\in \dot{H}_{\rad}^{-1}(\rthree) \). As in the definition of \( F^\omega \) in \eqref{in:eq_random_linear} , we assume that \( P_{\leq 2^5} f_0 = P_{\leq 2^5} g_0 = 0 \). We recall from \eqref{intro:eq_forced_nlw} that the low-frequency component of \( (f^\omega,g^\omega) \) will be treated as the initial data of the nonlinear component \(  v \). In order to use Littlewood-Paley theory in the spatial variables, it is convenient to introduce a second solution \( \widetilde{F} \) to the linear wave equation. A short computation shows that 
\begin{equation*}
\partial_t F = |\nabla| \left(\cos(t|\nabla|) |\nabla|^{-1} g + \frac{\sin(t|\nabla|)}{|\nabla|} (-|\nabla|f)\right) 
\end{equation*}
Then, 
\begin{equation}\label{bootstrap:eq_f_tilde}
\widetilde{F}:= \cos(t|\nabla|) |\nabla|^{-1} g + \frac{\sin(t|\nabla|)}{|\nabla|} (-|\nabla|f)
\end{equation}
satisfies \( \partial_t F = |\nabla| \ftil \) and has initial data \( \ftil|_{t=0} = |\nabla|^{-1} g \in L_\rad^2(\rthree) \) and \( \partial_t \ftil |_{t=0} = -|\nabla|f \in \dot{H}_{\rad}^{-1}(\rthree) \). After localizing in frequency space, we write
\begin{equation}\label{bootstrap:eq_ftil_in_out}
|\nabla|\ftilN(t,x) = \frac{1}{|x|} \left( \wout[|\nabla|\ftilN](t-r) + \win[|\nabla|\ftilN](t+r) \right)
\end{equation}

In the bootstrap argument, we want to apply the interaction flux estimate to the Littlewood-Paley pieces \( P_K v \) of \( v \). In order to deal with the operators \( P_K \), we need to slightly modify the weights. Unfortunately, we cannot use the Hardy-Littlewood maximal function, since it is unbounded in \( L^1 \). Instead, we define for each \( K \in 2^{\mathbb{N}} \) the operator 
\begin{equation}
S_K w = K \langle K \tau \rangle^{-2}  * w~.
\end{equation}
\begin{definition}[Interaction Flux Term]
Let \( (f_0,g_0)\in L_{\rad}^2(\rthree)\times \dot{H}_{\rad}^{-1}(\rthree) \) and assume that \( P_{\leq 2^5}f_0= P_{\leq 2^5} g_0 = 0 \). Let \( F \) be the solution of the linear wave equation with data \( (f_0,g_0) \), let \( \widetilde{F} \) be as in \eqref{bootstrap:eq_f_tilde}, let   \( v \) be a solution to \eqref{flux:eq_forced_nlw}, and let \( I \subseteq \mathbb{R} \). 
For \( *\in \{\text{out},\text{in} \} \), we define
\begin{align}
\mathcal{F}_{I,*}  &:= \sum_{N\geq 1} (N^{-\frac{1}{6\gamma}+2\delta}+N^{-2+2\delta}) \sup_{K\in 2^{\mathbb{Z}}} \| w_{*,K,N}(t-|x|)^{\frac{1}{6}} v(t,x) \|_{L_{t,x}^6(I\times \mathbb{R}^3)}^6\label{bootstrap:eq_flux_tilde} \\
&+ \sum_{N\geq 1} (N^{-\frac{1}{6\gamma}+2\delta}+N^{-2+2\delta}) \sup_{K\in 2^{\mathbb{Z}}} \| w_{*,\nabla,K,N}(t-|x|)^{\frac{1}{6}} v(t,x) \|_{L_{t,x}^6(\sI)}^6 \label{bootstrap:eq_flux_grad}\\
&+ \| W_*[F](t-|x|)^{\frac{1}{3}} v\|_{L_{t,x}^6(\sI)}^6		\label{bootstrap:eq_flux_F} ~, 
\end{align}
where \( w_{*,K,N}=  S_K( |\wstarf|^2)  \) and \( w_{*,\nabla,K,N}= S_K(|\wstargradfN|^2 ) \), see  Section \ref{sec:decomp}. 
For notational convenience, we also set 
\begin{equation*}
\mathcal{F}_I := \mathcal{F}_{I,\text{out}} + \mathcal{F}_{I,\text{in}}~. 
\end{equation*}
\end{definition}

In the following definition, we introduce two auxiliary norms on \( F \) that will be used in the rest of this paper.
\begin{definition}[{\(Y_I\) and \( Z\)-norms}]\label{bootstrap:def_auxiliary_norms}
Let \( (f_0,g_0)\in L_{\rad}^2(\rthree)\times \dot{H}_{\rad}^{-1}(\rthree) \) and assume that \( P_{\leq 2^5}f_0= P_{\leq 2^5} g_0 = 0 \). Let \( F \) be the solution of the linear wave equation with data \( (f_0,g_0) \), let \( \widetilde{F} \) be as in \eqref{bootstrap:eq_f_tilde}, and let \( I \subseteq \mathbb{R} \). Then, we define
\begin{align*}
\| F \|_{Y(I)} &:= \| N^{-\frac{3}{4} +\frac{1}{24\gamma} + \delta }  |x|^{\frac{3}{8}} |\nabla|\ftilN \|_{\ell_N^{\frac{8}{3}}L_t^{\frac{8}{3}}L_x^{\infty}(2^{\mathbb{N}}\times\sI)} \\
	&~+\| N^{-\frac{3}{4} + \frac{1}{24\gamma} + \frac{5\delta}{2}}  |x|^{\frac{3+2\delta}{8}} |\nabla| F_N\|_{\ell_N^{\frac{8}{3-2\delta}}L_t^{\frac{8}{3-2\delta}}L_x^{\frac{2}{\delta}}(2^{\mathbb{N}}\times\sI)} \\
	&~+ \| N^{-1+\delta} |x|^{-\frac{1}{6}} |\nabla| F_N \|_{\ell_N^6 L_t^6L_x^6(2^{\mathbb{N}}\times\sI)} 
			 +\|  N^{-1+\delta}  |x|^\frac{2}{3} |\nabla| \ftilN \|_{\ell_N^{12} L_t^{12} L_x^{12}(2^{\mathbb{N}}\times\sI)}\\
			&~+ \| |x|^{\frac{1}{4}} F \|_{L_t^4 L_x^\infty(\sI)} 
			~+ \| F \|_{L_t^5 L_x^{10}(\sI)}  
			~+ \| |x|^{-\frac{1}{6}} F \|_{L_{t,x}^6(\sI)} 
			~+ \| |x|^{\frac{2}{3}} F \|_{L_{t,x}^{12}(\sI)}~.
\end{align*}
Furthermore, we also define 
\begin{align}
\| F \|_Z &:=\sum_{*\in \{ \text{out},\text{in}\} }\sum_{p\in \{2,4,24\} } \| (N^{-\frac{1}{12\gamma}+2\delta}+N^{-1+\delta} ) \wstarf \|_{\ell_N^1 L_\tau^p(2^{\mathbb{N}}\times \mathbb{R})} \notag \\
&~+ \sum_{*\in \{ \text{out},\text{in}\} }\sum_{p\in \{2,4,24\} }\| (N^{-\frac{1}{12\gamma}+2\delta}+N^{-1+\delta} )  \wstargradfN\|_{\ell_N^1L_\tau^p(2^\mathbb{N}\times\mathbb{R})}			  \notag \\
&~+ \sum_{*\in \{ \text{out},\text{in} \} } \sum_{p\in \{ 2,4,24\}} \| W_*[F] \|_{L_\tau^p(\mathbb{R})}  \notag \\
&~+  \| N^{\delta}  |x|^{\frac{1}{2}} F_N \|_{\ell_N^1 L_t^\infty L_x^\infty(2^{\mathbb{N}}\times \mathbb{R}\times \rthree)} + \| F \|_{L_t^\infty L_x^6(\mathbb{R}\times \rthree)} ~. \notag
\end{align}
\end{definition}

We remark that \( \| F \|_{Y_I} \) is divisible in space-time.  More precisely, let \( \eta > 0 \) be given and assume that  \( \| F \|_{Y(\mathbb{R}) } < \infty \). Then, there exists a finite number \( J=J(\eta,\| F\|_{Y(\mathbb{R})}) \) and a partition of \( \mathbb{R} \) into finitely many intervals \( I_1,\hdots , I_J \)  such that \( \| F \|_{I_j} < \eta \) for all \( j=1,\hdots,J \). \\

\begin{lem}[Almost sure finiteness of \( Y \) and \( Z \)-norms]\label{bootstrap:lem_auxiliary_norms}
Let \( (f,g) \in H_{\rad}^s(\rthree) \times H_{\rad}^{s-1}(\rthree) \), let \( 0 < \gamma \leq 1 \), let \( s > \max(0,1 - \frac{1}{12\gamma})\), and let \( F^\omega \) be as in \eqref{in:eq_random_linear}. If \( \delta=\delta(s,\gamma) > 0 \) is chosen sufficiently small, we have that 
\begin{equation*}
\| F^\omega \|_{Y(\mathbb{R})} < \infty \qquad \text{and} \qquad \| F^\omega \|_Z < \infty \qquad \text{a.s.}~. 
\end{equation*}
\end{lem}
\begin{proof}
In the following, we assume that \( \delta= \delta(\gamma,s)>0 \) is sufficiently small.  In the computations below, we have that \( N \geq 2^6 \) and \( (t,x) \in \mathbb{R}\times \rthree \). 
For \( \sigma\geq 8/3 \), it follows from Minkowski's integral inequality and  Lemma \ref{prob:lem_prob_strichartz_p_infty} that
\begin{align*}
&~~~~\| N^{-\frac{3}{4}+ \frac{1}{24\gamma}+  \delta} |x|^{\frac{3}{8}} |\nabla|\ftilN^\omega \|_{L_\omega^\sigma \ell_{N}^{\frac{8}{3}} L_t^{\frac{8}{3}} L_x^\infty} \\
&\leq \| N^{-\frac{3}{4}+ \frac{1}{24\gamma}+  \delta} |x|^{\frac{3}{8}} |\nabla|\ftilN^\omega \|_{ \ell_N^{\frac{8}{3}}L_\omega^\sigma L_t^{\frac{8}{3}} L_x^\infty} \\
&\lesssim \sqrt{\sigma} \| N^{1-\frac{1}{12\gamma} + \delta} (f_N,g_N) \|_{\ell_N^{\frac{8}{3}}(L_x^2\times \dot{H}_x^{-1})} \\
&\lesssim \sqrt{\sigma} \| (f,g) \|_{H_{x}^s\times H_x^{s-1}} ~. 
\end{align*}
In particular, we have that  \begin{equation*}
\| N^{-\frac{3}{4}+ \frac{1}{24\gamma}+  \delta} |x|^{\frac{3}{8}} |\nabla|\ftilN^\omega \|_{\ell_{N}^{\frac{8}{3}} L_t^{\frac{8}{3}} L_x^\infty} < \infty \qquad 
\end{equation*}
almost surely. A similar argument for the remaining terms in the \( Y_{\mathbb{R}} \)-norm leads to the regularity restrictions
\begin{equation*}
s > \max\left( 1 -\tfrac{1}{12\gamma}, 1- \tfrac{1}{3\gamma}, \tfrac{1}{2}- \tfrac{5}{12\gamma}, 1- \tfrac{1}{4\gamma}, 1 -\tfrac{3}{10\gamma}, 1 - \tfrac{1}{3\gamma}, \tfrac{1}{2}- \tfrac{5}{12\gamma}\right)~,
\end{equation*}
which have been listed in the same order as the terms in the definition of \( \| F^\omega \|_{Y_\mathbb{R}} \). Next, we estimate \( \| F^\omega \|_Z \). Using Corollary \ref{decomp:cor_integrability}, the terms involving 
\( \| W_*[|\nabla| \ftilN ] \|_{L_\tau^p} \) lead to the restriction 
\begin{equation*}
s >  \max\left( (1-\tfrac{1}{\gamma}) (\tfrac{1}{2}-\tfrac{1}{24}), 0 \right) + \max\left( 1-\tfrac{1}{12\gamma},0 \right)~. 
\end{equation*}
Since \( 0 < \gamma \leq 1 \), this leads to \( s > \max(1-\frac{1}{12\gamma}, 0 ) \). Using Lemma \ref{prob:lem_prob_strichartz_q_infty}, the fourth and fifth summand in the \( Z\)-norm lead to the restriction
\begin{equation*}
s > \max\left( 1- \tfrac{1}{3\gamma}, 1- \tfrac{1}{2\gamma} \right)~. 
\end{equation*}
\end{proof}
In this paper, the condition \( \gamma \leq 1 \) is only used in the proof of Lemma \ref{bootstrap:lem_auxiliary_norms}. By changing the restriction on \( s \), we could also treat a slightly larger range of parameters \( \gamma \). 
 
\section{Control of error terms}\label{section:major}
In this section, we estimate the error terms in Proposition \ref{flux:prop_energy_increment}, Lemma \ref{flux:lem_morawetz_estimate}, and Proposition \ref{flux:prop_forward_interaction_flux}. Before we begin with our main estimates we prove an auxiliary lemma. 

\begin{lem}\label{major:eq_weighted_littlewood_paley}
Let \( w\in L^1_\tau(\mathbb{R}) \) be nonnegative. Let \( K \in 2^\mathbb{N} \) be arbitrary, and let \( S_K  \) be defined by
\begin{equation*}
S_K w = K ~ \langle K \rho \rangle^{-2} * w ~. 
\end{equation*}
 Then, we have for all \(v \in L_{\text{loc}}^1(\rthree) \) that
\begin{equation}\label{error:eq_LWP_estimate}
\int_{\rthree} |P_K v(x)|^6 w(t-|x|)  \dx \lesssim\int_{\rthree} |v(x)|^6 ( (S_K w)(t-|x|) + |x|^{-1} \| w \|_{L_\tau^1} ) \dx~. 
\end{equation}
\end{lem}
\begin{proof}
We prove \eqref{error:eq_LWP_estimate} by interpolation. The \( L^\infty \rightarrow L^\infty \) estimate is trivial. Thus, it suffices to prove the \( L^1 \rightarrow L^1 \) estimate 
\begin{equation}\label{error:eq_LWP_estimate_L1}
\int_{\rthree} |P_K v(x)| w(t-|x|)  \dx \lesssim \int_{\rthree} |v(x)| ( (S_K w)(t-|x|) + |x|^{-1} \| w \|_{L_\tau^1} ) \dx~. 
\end{equation}
Let \( \Psi \in \{ \phi, \psi \}  \) be as in the definition of the Littlewood-Paley projection.  Then, 
\begin{align*}
\int_{\rthree} |P_K v(x)| w(t-|x|) \dx &\leq \int_{\mathbb{R}^3} \int_{\rthree} |v(y)|~ K^3 |\widecheck{\Psi}(K(x-y))| w(t-|x|) \dy \dx \\
						&= \int_{\mathbb{R}^3} |v(y)| \left( K^3 \int_{\rthree} |\widecheck{\Psi}(K(y-x))| w(t-|x|) \dx \right) \dy~. 
\end{align*} 
Hence, we it remains to establish the pointwise bound
\begin{equation*}
K^3 \int_{\rthree} |\widecheck{\Psi}(K(y-x))| w(t-|x|) \dx \lesssim (S_K * w)(t-|y|) + |y|^{-1} \| w \|_{L_\tau^1} ~. 
\end{equation*}
Now, the main task consists of converting the left-hand side into a one-dimensional integral. Using an integral formula from \cite[p. 8]{Sogge08}, we have that 
\begin{align}
&K^3 \int_{\rthree} |\widecheck{\Psi}(K(y-x))| w(t-|x|) \dx  \notag  \\
&= K^3 \int_{\rthree} |\widecheck{\Psi}(Kx)| w(t-|y-x|) \dx  \notag \\
&\lesssim K^3 \int_{0}^\infty |\widecheck{\Psi}(Kr)| \left( \int_{|x|=r} w(t-|y-x|) \dsigma  \right) \dr \notag  \\
&= K^3 \int_{0}^\infty |\widecheck{\Psi}(Kr)| \left( \int_{|y-x|=r} w(t-|x|) \dsigma  \right) \dr \notag \\
&= K^3 \int_{0}^\infty |\widecheck{\Psi}(Kr)| \frac{2\pi r}{|y|} \int_{||y|-r|}^{|y|+r} w(t-\rho) \rho \drho \dr \notag \\
&\lesssim \frac{K^3}{|y|} \int_{0}^{4|y|} \int_{|y|-r}^{|y|+r}  r |\widecheck{\Psi}(Kr)| w(t-\rho)|\rho| \drho \dr  \label{error:eq_LWP_estimate_proof_1}\\
&~~~~+  \frac{K^3}{|y|} \int_{4|y|}^\infty \int_{r-|y|}^{r+|y|}  r |\widecheck{\Psi}(Kr)| w(t-\rho)\rho \drho \dr \notag
\end{align}
Let us now estimate the first summand in \eqref{error:eq_LWP_estimate_proof_1}. We have that 
\begin{align*}
&\frac{K^3}{|y|} \int_{0}^{4|y|} \int_{|y|-r}^{|y|+r}  r |\widecheck{\Psi}(Kr)| w(t-\rho)|\rho| \drho \dr \\
&= \frac{K^3}{|y|} \int_{0}^{4|y|} \int_{-r}^{r} r |\widecheck{\Psi}(Kr) w(t-|y|-\rho) |(|y|+\rho)| \drho \dr \\
&\lesssim K^3 \int_{0}^{4|y|} \int_{-r}^r  r |\widecheck{\Psi}(Kr)| w(t-|y|-\rho) \drho \dr \\
&\leq K^3 \int_{-\infty}^{\infty} \left( \int_{|\rho|}^\infty |\widecheck{\Psi}(Kr)| r \dr \right) w(t-|y|-\rho) \drho \\
&\leq K \int_{-\infty}^\infty \left( \int_{K|\rho|}^\infty |\widecheck{\Psi}(r)| r \dr \right) w(t-|y|-\rho) \drho \\
&\lesssim K \int_{-\infty}^{\infty} \langle K |\rho| \rangle^{-2} w(t-|y|-\rho) \drho \\
&= (S_K w)(t-|y|)~. 
\end{align*}
Thus, it remains to estimate the second integral in \eqref{error:eq_LWP_estimate_proof_1}. We have that 
\begin{align*}
 &\frac{K^3}{|y|} \int_{4|y|}^\infty \int_{r-|y|}^{r+|y|}  r |\widecheck{\Psi}(Kr)| w(t-\rho)\rho \drho \dr\\
 &\lesssim \frac{K^3}{|y|} \int_{4|y|}^\infty \int_{r-|y|}^{r+|y|} r^2 |\widecheck{\Psi}(Kr)|w(t-\rho) \drho \dr \\
 &\leq \frac{K^3}{|y|}\| w \|_{L_\tau^1(\mathbb{R})}  \int_{0}^\infty |\widecheck{\Psi}(Kr)| r^2 \dr \\
 &\leq \frac{1}{|y|} \| w \|_{L_\tau^1(\mathbb{R})} \int_{0}^\infty |\widecheck{\Psi}(r)| r^2 \dr \\
 &\lesssim \frac{1}{|y|} \| w \|_{L_\tau^1(\mathbb{R})}~. 
\end{align*}
\end{proof}

\begin{cor}[Frequency-Localized Interaction Flux Estimate]\label{major:cor_freq_flux}
Let \( F \) be as in Definition \ref{bootstrap:def_auxiliary_norms} and let \( v\colon I \times \rthree \rightarrow \mathbb{R} \) be a solution of \eqref{flux:eq_forced_nlw}. Then, we have that 
\begin{equation}\label{major:eq_frequency_localized_interaction}
\sup_{K\in 2^{\mathbb{N}}} \| |x|^{\frac{1}{3}} (|\nabla|\widetilde{F}_N)^{\frac{1}{3}} P_K v \|_{L_{t,x}^6(\sI)}^6 \lesssim \min\left(N^{\frac{1}{6\gamma}-2\delta},N^{2-2\delta}\right) \left( \mathcal{F}_I + \| F\|_Z^2 \mathcal{A}_I \right)~. 
\end{equation}
\end{cor}

\begin{rem} The flux estimate yields much better integrability in the spatial variable \( x \) than the Morawetz estimate. To see this, note that \eqref{major:eq_frequency_localized_interaction} cannot be controlled by the Morawetz term. For instance, one might try to estimate
\begin{equation*}
\| |x| |\nabla|\ftilN v^3 \|_{L_{t,x}^2(\sI)} \lesssim \| |x|^{\frac{3}{2}} |\nabla|\ftilN \|_{L_{t,x}^\infty(\sI)}~ \| |x|^{-\frac{1}{6}} v \|_{L_{t,x}^6(\sI)}^3~.
\end{equation*}
Even for smooth and compactly supported initial data,  \( |\nabla|\ftilN \) only decays like \( \sim (1+|t|)^{-1} \) and is morally supported around the light cone \( |x|= |t| \). Thus, the term \( \| |x|^{\frac{3}{2}} |\nabla|\ftilN \|_{L_{t,x}^\infty(\sI)} \) grows like \( \sim (1+|t|)^{\frac{1}{2}} \) as \( I \) increases.
\end{rem}
\begin{proof}
Using the in/out-decomposition and Lemma \ref{major:eq_weighted_littlewood_paley}, it follows that 
\begin{align*}
&\quad ~\| |x|^{\frac{1}{3}} \big( |\nabla| \widetilde{F}_N \big)^{\frac{1}{3}} P_K v \|_{L_{t,x}^6(\sI)}^6 \\
&\lesssim \| (\woutf)^{\frac{1}{3}} P_K v \|_{L_{t,x}^6(\sI)}^6 + \| (\winf)^{\frac{1}{3}} P_K v \|_{L_{t,x}^6(\sI)}^6\\
&\lesssim \| S_K( |\woutf|^2)^{\frac{1}{6}}  v \|_{L_{t,x}^6(\sI)}^6 + \| |\woutf|^2 \|_{L_\tau^1} \| |x|^{-\frac{1}{6}} v \|_{L_{t,x}^6(\sI)}^6 \\
&\quad+ \| S_K( |\winf|^2)^{\frac{1}{6}}  v \|_{L_{t,x}^6(\sI)}^6 + \| |\winf|^2 \|_{L_\tau^1} \| |x|^{-\frac{1}{6}} v \|_{L_{t,x}^6(\sI)}^6 \\
&\lesssim \min\left( N^{\frac{1}{6\gamma}-2\delta}, N^{2-2\delta} \right) \left( \mathcal{F}_I + \| F \|_Z^2 \mathcal{A}_I \right)~.
\end{align*}
By taking the supremum over \( K \in 2^\mathbb{N} \), we arrive at \eqref{major:eq_frequency_localized_interaction}.
\end{proof}

\subsection{Energy increment}\label{section:energy_increment}
In this section, we control the main error term in the energy increment. 
\begin{prop}[Main error term in energy increment]\label{major:prop_energy_increment}
Let \( F \) be as in Definition \ref{bootstrap:def_auxiliary_norms} and let \( v\colon I \times \rthree \rightarrow \mathbb{R} \) be a solution of \eqref{flux:eq_forced_nlw}. Then, it holds that 
\begin{equation}\label{major:eq_energy_increment}
\left| \int_{I} \int_{\mathbb{R}^3} (|\nabla|\ftil) v^5 \dx \dt \right| \lesssim ( \mathcal{F}_{I}+ \| F\|_Z^2 \mathcal{A}_I )^{\frac{1}{6}} \mathcal{A}_I^\frac{7}{12} \mathcal{E}_I^{\frac{1}{4}} \| F \|_{Y_I}^{\frac{2}{3}} ~. 
\end{equation}
\end{prop}
\begin{rem}\label{major:rem_regularity}
Instead of using \( \mathcal{F}_{I}^{\frac{1}{6}} \) to overcome the logarithmic divergence, we could also just use \( \mathcal{F}_I^{\epsilon} \). Then, the term 
\(     \| |x|^{\frac{3}{8}} |\nabla|\ftilN \|_{\ltx{\frac{8}{3}}{\infty}(\sI)} \) changes into a (non-endpoint) term 
\( \| |x|^{\frac{1}{4}-} |\nabla| \ftilN \|_{\ltx{4-}{\infty}(\sI)} \). The probabilistic gain should then increase from \( \frac{2}{3}\cdot \frac{1}{8\gamma} \) to \( \frac{1}{4\gamma} \) derivatives, which should lead to the restriction \( s> \max( 1-\frac{1}{4\gamma}, 0) \). For expository purposes, we do not present this argument here. 
\end{rem}
\begin{proof}
Using a Littlewood-Paley decomposition, we write \( v = \sum_{K\geq 1} P_K v \) and \( \ftil = \sum_{N\geq 2^6} \ftilN  \). Thus, 
\begin{align*}
\left| \int_{I} \int_{\mathbb{R}^3} (|\nabla|\ftil) v^5 \dx \dt \right| 
&\lesssim \sum_{N\geq 2^6} \sum_{K_1\geq K_2 \geq \hdots\geq K_5 \geq 1} \left| \int_{I} \int_{\mathbb{R}^3} (|\nabla|\ftilN) \prod_{j=1}^5 P_{K_j} v\dx \dt \right| \\
&= \sum_{N\geq 2^6} \sum_{ \substack{K_1\geq K_2 \geq \hdots \geq K_5 \geq 1 \\ K_1 \geq 2^{-4}N}} \left| \int_{I} \int_{\mathbb{R}^3} (|\nabla|\ftilN) \prod_{j=1}^5 P_{K_j} v\dx \dt \right| 
\end{align*}
Note that, for all summands above, we have \( K_1> 1 \). Using Proposition \ref{prelim:prop_mikhlin} and Corollary \ref{major:cor_freq_flux}, it follows that 
 \begin{align*}
  &\left| \int_{I} \int_{\mathbb{R}^3} (|\nabla|\ftilN) \prod_{j=1}^5 P_{K_j} v\dx \dt \right|  \\
  &\leq \| |x|^{\frac{3}{8}} |\nabla| \ftilN \|_{\ltx{\frac{8}{3}}{\infty}(\sI)}^{\frac{2}{3}} \| |x|^{\frac{1}{3}} (|\nabla|\ftilN)^{\frac{1}{3}} P_{K_5} v \|_{L_{t,x}^6(\sI)}~ \prod_{j=2}^4 \| |x|^{-\frac{1}{6}} P_{K_j} v \|_{L_{t,x}^6(\sI)} \\
  &\quad\cdot \| |x|^{-\frac{1}{6}} P_{K_1} v \|_{L_{t,x}^6(\sI)}^{\frac{1}{2}} \| P_{K_1} v \|_{L_t^\infty L_x^2 (\sI)}^{\frac{1}{2}} \\
&\lesssim N^{\frac{2}{3} \left( \frac{3}{4}- \delta - \frac{1}{24\gamma} \right)} \| F \|_{Y_I}^{\frac{2}{3}} N^{\frac{1}{36\gamma}-\frac{\delta}{3}} (\mathcal{F}_I + \| F\|_Z^2 \mathcal{A}_I )^{\frac{1}{6}}  \mathcal{A}_I^{\frac{7}{12}} K_1^{-\frac{1}{2}} \mathcal{E}_I^{\frac{1}{4}} \\
&= \left( \frac{N}{K_1} \right)^{\frac{1}{2}-\delta} K_1^{-\delta}  \| F \|_{Y_I}^{\frac{2}{3}} (\mathcal{F}_I + \| F\|_Z^2 \mathcal{A}_I )^{\frac{1}{6}}  \mathcal{A}_I^{\frac{7}{12}} \mathcal{E}_I^{\frac{1}{4}}~.
 \end{align*}
Using that \( K_1 \gtrsim N \) and \( K_1, \hdots, K_5 \geq 1 \), we obtain \eqref{major:eq_energy_increment} after summing. 
\end{proof}

\subsection{Morawetz estimate}\label{sec:morawetz_error}
In this section, we control the main error term in the Morawetz estimate. The main new difficulty is the weight \( x/|x| \). 
\begin{prop}[Main error term in Morawetz estimate]\label{major:prop_morawetz}
Let \( F \) be as in Definition \ref{bootstrap:def_auxiliary_norms} and let \( v \) be a solution of \eqref{flux:eq_forced_nlw}. Then, 
\begin{align*}
&\left| \int_I \int_{\mathbb{R}^3} \frac{x}{|x|}\cdot \nabla_x(F)~   v^5 \dx \dt \right| \\
&\lesssim  ( \mathcal{F}_I+ \| F\|_Z^2 \mathcal{A}_I )^{\frac{1}{6}} \mathcal{A}_I^{\frac{7}{12}+\frac{\delta}{6}} \mathcal{E}_I^{\frac{1}{4}-\frac{\delta}{2}} 
	 \| F \|_{Y_I}^\frac{2}{3} + \| F \|_{Y_I} \mathcal{A}_I^{\frac{5}{6}} ~. 
\end{align*}
\end{prop}
\begin{proof}
As before, we use a Littlewood-Paley decomposition and write
\begin{equation*}
\left| \int_I \int_{\mathbb{R}^3} \frac{x}{|x|}\cdot \nabla_x(F)~   v^5 \dx \dt \right| \lesssim \sum_{N\geq 2^5} \sum_{\substack{ L\geq 1, K_1 \geq \hdots \geq K_5 \geq 1\\ \max(L,K_1)\geq 2^{-4}N}} 
\left| \int_I \int_{\mathbb{R}^3} P_L\left(\frac{x}{|x|} \right)\cdot \nabla_x(F_N)~   \prod_{j=1}^5 P_{K_j} v ~ \dx \dt \right|
\end{equation*}
\textbf{Case 1: \( K_1 \geq L \)}. From the conditions \( K_1 \geq 2^{-4} N \) and \( N\geq 2^5 \), it follows that \( K_1 > 1 \). Thus, we can place \( P_{K_1} v \) in \( L_t^\infty L_x^2(\sI) \). 
Using \eqref{decomp:eq_grad_decomp}, we estimate
\begin{align}
&\left| \int_I \int_{\mathbb{R}^3} P_L\Big(\frac{x}{|x|} \Big)\cdot \nabla_x(F_N)~   \prod_{j=1}^5 P_{K_j} v ~ \dx \dt \right| \notag \\
  &\leq  \int_{I} \int_{\mathbb{R}^3} \left|P_L\Big(\frac{x}{|x|} \Big)\right| ~\frac{1}{|x|^{\frac{1}{3}}} \left( |\woutgradfN|+ |\wingradfN| \right)^{\frac{1}{3}} |\nabla_xF_N|^{\frac{2}{3}} \prod_{j=1}^5 |P_{K_j} v|\dx \dt  \label{major:eq_morawetz_main}\\
  &+ \int_{I} \int_{\mathbb{R}^3} \left|P_L\Big(\frac{x}{|x|} \Big)\right|~ \frac{1}{|x|^{\frac{1}{3}}} |F_N|^{\frac{1}{3}} |\nabla_xF_N|^{\frac{2}{3}} \prod_{j=1}^5 |P_{K_j} v|\dx \dt\notag 
\end{align}
To control the first term, we estimate 
\begin{align*}
& \int_{I} \int_{\mathbb{R}^3} \left|P_L\Big(\frac{x}{|x|} \Big)\right| ~\frac{1}{|x|^{\frac{1}{3}}} \left( |\woutgradfN|+ |\wingradfN| \right)^{\frac{1}{3}} |\nabla_xF_N|^{\frac{2}{3}} \prod_{j=1}^5 |P_{K_j} v|\dx \dt \\
&\lesssim \| P_L\Big( \frac{x}{|x|} \Big) \|_{L_{t,x}^\infty(\sI)} \left( \| |\woutgradfN|^{\frac{1}{3}} P_{K_5} v \|_{L_{t,x}^6(\sI)} +\| |\wingradfN|^{\frac{1}{3}} P_{K_5} v \|_{L_{t,x}^6(\sI)} \right)\\
&\cdot \prod_{j=2}^4 \| |x|^{-\frac{1}{6}} P_{K_j} v \|_{L_{t,x}^6(\sI)} \cdot \| |x|^{-\frac{1}{6}} P_{K_1} v \|_{L_{t,x}^6(\sI)}^{\frac{1}{2}+\delta} \| P_{K_1} v \|_{L_t^\infty L_x^2(\sI)}^{\frac{1}{2}-\delta}
\| |x|^{\frac{3+2\delta}{8}} \nabla_x F_N\|_{L_t^{\frac{8}{3-2\delta}}L_x^{\frac{2}{\delta}}(\sI)}^{\frac{2}{3}} 
\end{align*}
The first factor is estimated by
\begin{equation*}
\| P_L( \frac{x}{|x|} ) \|_{L_{t,x}^\infty(\sI)} \lesssim \| \frac{x}{|x|}  \|_{L_{t,x}^\infty(\mathbb{R}\times \rthree)} \lesssim 1 ~. 
\end{equation*}
Using Lemma \ref{major:eq_weighted_littlewood_paley} and arguing as in the proof of Corollary \ref{major:cor_freq_flux}, we estimate the second factor by
\begin{align*}
 &\| |\woutgradfN|^{\frac{1}{3}} P_{K_5} v \|_{L_{t,x}^6(\sI)} +\| |\wingradfN|^{\frac{1}{3}} P_{K_5} v \|_{L_{t,x}^6(\sI)} \\
 &\lesssim  \| S_{K_5}(|\woutgradfN|^2)^\frac{1}{6}  v \|_{L_{t,x}^6(\sI)} +\| S_{K_5}(|\wingradfN|^2)^{\frac{1}{6}} v \|_{L_{t,x}^6(\sI)} \\
 &+ \left( \| \woutgradfN \|_{L_\tau^2(\mathbb{R})}^\frac{1}{3} + \| \wingradfN \|_{L_\tau^2(\mathbb{R})}^\frac{1}{3} \right)  \| |x|^{-\frac{1}{6}} v \|_{L_{t,x}^6(\sI)}   \\
 &\lesssim N^{\frac{1}{36\gamma}-\frac{\delta}{3}} ( \mathcal{F}_I + \| F \|_Z^2 \mathcal{A}_I )^{\frac{1}{6}} 
\end{align*}
From Proposition \ref{prelim:prop_mikhlin}, we have that 
\begin{equation*}
\| |x|^{-\frac{1}{6}} P_{K_j} v \|_{L_{t,x}^6(\sI)} \lesssim \| |x|^{-\frac{1}{6}} v \|_{L_{t,x}^6(\sI)} \lesssim \mathcal{A}_I^{\frac{1}{6}}~. 
\end{equation*}
Furthermore, since \( K_1 > 1 \), we have that 
\begin{equation*}
\| P_{K_1} v \|_{L_t^\infty L_x^2(\sI)} \lesssim K_1^{-1} \mathcal{E}_I^{\frac{1}{2}}
\end{equation*}
Finally, applying Proposition \ref{prelim:prop_mikhlin} to the Riesz multipliers, we have that 
\begin{equation*}
\| |x|^{\frac{3+2\delta}{8}} \nabla_x F_N\|_{L_t^{\frac{8}{3-2\delta}(\sI)}L_x^{\frac{2}{\delta}}(\sI)} \lesssim \| |x|^{\frac{3+2\delta}{8}(\sI)} |\nabla| F_N\|_{L_t^{\frac{8}{3-2\delta}}L_x^{\frac{2}{\delta}}(\sI)} \lesssim 
N^{\frac{3}{4}-\frac{1}{24\gamma} - \frac{5\delta}{2}} \| F \|_{Y_I}~. 
\end{equation*}
Putting everything together, it follows that 
\begin{align*}
&\int_{I} \int_{\mathbb{R}^3} \left|P_L\left(\frac{x}{|x|} \right)\right| ~\frac{1}{|x|^{\frac{1}{3}}} \left( |\woutgradfN|+ |\wingradfN| \right)^{\frac{1}{3}} |\nabla_xF_N|^{\frac{2}{3}} \prod_{j=1}^5 |P_{K_j} v|\dx \dt  \\
&\lesssim \left( \frac{N}{K_1} \right)^{\frac{1}{2}-2\delta} K_1^{-\delta} 
	 \left(\mathcal{F}_I + \| F \|_Z^2 \mathcal{A}_I \right)^{\frac{1}{6}} \mathcal{A}_I^{\frac{7}{12}+\frac{\delta}{6}} \mathcal{E}_I^{\frac{1}{2}-\frac{\delta}{2}} 
	 \| F \|_{Y_I}^\frac{2}{3}~. 
\end{align*}
Using the decay \( K_1^{-\delta} \) in the highest frequency, we may sum \( N, L, K_1,\hdots K_5 \).  \\
Next, we estimate the second term in \eqref{major:eq_morawetz_main}. We have that 
\begin{align*}
& \int_{I} \int_{\mathbb{R}^3} \left|P_L\left(\frac{x}{|x|} \right)\right|~ \frac{1}{|x|^{\frac{1}{3}}} |F_N|^{\frac{1}{3}} |\nabla_xF_N|^{\frac{2}{3}} \prod_{j=1}^5 |P_{K_j} v|\dx \dt \\
 &\lesssim \| P_L( \frac{x}{|x|}) \|_{L_{t,x}^\infty(\sI)} ~ \prod_{j=2}^5 \| |x|^{-\frac{1}{6}} P_{K_j} v \|_{L_{t,x}^6(\sI)}  \cdot \| |x|^{-\frac{1}{6}} P_{K_1} v \|_{L_{t,x}^6(\sI)}^{\frac{1}{2}+\delta} \| P_{K_1} v \|_{L_t^\infty L_x^2(\sI)}^{\frac{1}{2}-\delta}\\
 &\cdot   \| |x|^{\frac{1}{2}} F_N \|_{L_{t,x}^\infty(\sI)}^{\frac{1}{3}} \| |x|^{\frac{3+2\delta}{8}} \nabla_x F_N \|_{L_t^{\frac{8}{3-2\delta}}L_x^{\frac{2}{\delta}}(\sI)}^{\frac{2}{3}}~. 
\end{align*}
Arguing as above, together with
\( \| |x|^{\frac{1+\delta}{2}} F_N \|_{L_{t,x}^\infty(\sI)} \leq N^{-\delta} \| F \|_Z ~,\)
we get that 
\begin{align*}
& \int_{I} \int_{\mathbb{R}^3} \left|P_L\left(\frac{x}{|x|} \right)\right|~ \frac{1}{|x|^{\frac{1}{3}}} |F_N|^{\frac{1}{3}} |\nabla_xF_N|^{\frac{2}{3}} \prod_{j=1}^5 |P_{K_j} v|\dx \dt \\ 
&\lesssim N^{\frac{1}{2}-\frac{1}{36\gamma}- 2\delta} K_1^{-\frac{1}{2}+\delta} \mathcal{A}_I^{\frac{3}{4}+\frac{\delta}{6}} \mathcal{E}_I^{\frac{1}{4}-\delta} \| F \|_{Y_I}^{\frac{2}{3}} \| F \|_Z^{\frac{1}{3}}\\
&\lesssim \left( \frac{N}{K_1} \right)^{\frac{1}{2}-2\delta} K_1^{-\delta} \mathcal{A}_I^{\frac{3}{4}+\frac{\delta}{6}} \mathcal{E}_I^{\frac{1}{4}-\frac{\delta}{2}} \| F \|_{Y_I}^{\frac{2}{3}} \| F \|_Z^{\frac{1}{3}}~. 
\end{align*}
Summing over the appropriate range, this contribution is acceptable.\\
\textbf{Case 2: \( L \geq K_1 \)}. Consequently, we have that \( L \geq 2^{-4} N > 1 \). Using Lemma \ref{prelim:lem_lwp_pointwise}, it follows that 
\( |P_L( \frac{x}{|x|})| \lesssim (L|x|)^{-1} \). This yields
\begin{align*}
&\left| \int_I \int_{\mathbb{R}^3} P_L\left(\frac{x}{|x|} \right)\cdot \nabla_x(F_N)~   \prod_{j=1}^5 P_{K_j} v ~ \dx \dt \right|  \\
&\lesssim L^{-1}  \int_I \int_{\mathbb{R}^3} \frac{1}{|x|} |\nabla_x(F_N)|~   \prod_{j=1}^5 |P_{K_j} v| ~ \dx \dt \\
&\lesssim L^{-1} \| |x|^{-\frac{1}{6}} |\nabla| F_N \|_{L_{t,x}^6} \prod_{j=1}^5 \| |x|^{-\frac{1}{6}} P_{K_j} v \|_{L_{t,x}^6} \\
&\lesssim \left( \frac{N}{L} \right)^{1-\delta} L^{-\delta} \| F \|_{Y_I} \mathcal{A}_I^{\frac{5}{6}} 
\end{align*}
Using the decay \( L^{-\delta} \) in the highest frequency, we may sum \( N, L, K_1,\hdots K_5 \).  \\
\end{proof}

\subsection{Interaction flux estimate} \label{section:interaction_flux}
In this section, we control the main error terms in the interaction flux estimate. The main difficulty is the weight \( \int_{-\infty}^{t-|x|} w(\tau)\dtau \). First, we recall a radial Sobolev embedding.

\begin{lem}\label{lem_radial_sobolev} For any  \( v \in L_t^\infty\dot{H}_{\rad}^1(\sI) \), we have 
\begin{equation*}
\sup_{K\in 2^{N}}\| |x|^{\frac{1}{2}} P_K v \|_{L_{t,x}^\infty(\sI)} \lesssim \| v \|_{L_t^\infty L_x^6(\sI)}^\frac{3}{4} \| \nabla v \|_{L_t^\infty L_x^2(\sI)}^{\frac{1}{4}} \lesssim \sup_{t\in I}E[v](t)^{\frac{1}{4}}~.
\end{equation*}
\end{lem}
\begin{proof}
Let \( r \in \mathbb{R}_{>0} \). Then, we have that 
\begin{align*}
(P_K v)^4(t,r) &= 4 \int_{r}^\infty (P_Kv)^3(t,\rho) (\partial_r P_Kv)(t,\rho) \drho\\
		&\leq 4r^{-2} \int_{r}^\infty |(P_K v)^3(t,\rho)| |\partial_r P_K v(t,\rho)| \rho^2 \drho \\
		&\leq 4 r^{-2} \| P_K v(t,x) \|_{L_x^6(\rthree)}^{3} \| \nabla P_K v(t,x) \|_{L_x^2(\rthree)} \\
		&\leq 4 r^{-2} \| v(t,x) \|_{L_x^6(\rthree)}^{3} \| \nabla v(t,x) \|_{L_x^2(\rthree)}~. 
\end{align*}
The first inequality then follows by taking the supremum in \(r \) and \( t \). The second inequality follows from the definition of \( E[v]\). 
\end{proof}

\begin{prop}[First main error term in interaction flux estimate]\label{major:prop_interaction_flux}
Let \( w \in L^1_\tau(\mathbb{R}) \cap L_\tau^{12}(\mathbb{R}) \) be a nonnegative weight. Let \( F \) be as in Definition \ref{bootstrap:def_auxiliary_norms} and let \( v\colon I \times \rthree \rightarrow \mathbb{R} \) be a solution of \eqref{flux:eq_forced_nlw}. Then, it holds that
\begin{align*}\label{major:eq_flux_estimate}
&\left| \int_{I} \int_{\mathbb{R}^3} \left( \int_{-\infty}^{t-|x|} w \dtau \right) (|\nabla| \widetilde{F}) v^5 \dx \dt \right|  \\
&\lesssim \| w \|_{L_\tau^1(\mathbb{R})} \| F \|_{Y_I}^{\frac{2}{3}}  ( \mathcal{F}_{I} + \| F\|_Z^2 \mathcal{A}_I )^{\frac{1}{6}} \mathcal{A}_I^\frac{7}{12} \mathcal{E}_I^{\frac{1}{4}} \\
&+  \| w \|_{L_\tau^2(\mathbb{R})}
 \left( \mathcal{F}_{I} +\| F \|_{Z}^2 \mathcal{A}_I  \right)^{\frac{1}{2}} 
 \mathcal{E}_I^{\frac{1}{2}}\\
 &+\| w \|_{L_\tau^{12}(\mathbb{R})} \| F \|_{Y_I}  ~ \mathcal{A}_I^{\frac{5}{6}} ~.
\end{align*}
\end{prop}
The same argument also controls the main error term in the backward interaction flux estimate. 
\begin{proof}
As before, we use Littlewood-Paley theory to decompose into frequency-localized functions. Then, it remains to control 
\begin{equation*}
 \sum_{N\geq2^{6}} \sum_{\substack{L\geq1, K_1\geq \hdots \geq K_5 \geq 1 \\ \max(L,K_1) \geq 2^{-4} N} } 
\left| \int_{I} \int_{\mathbb{R}^3} P_L\left( \int_{-\infty}^{t-|x|} w \dtau \right) (|\nabla| \ftilN) \prod_{j=1}^5 P_{K_j} v ~ \dx \dt \right|~.
\end{equation*}
We distinguish several different cases.\\
\textbf{Case 1: \( K_1 \geq L \).} We have that 
\begin{align*}
&\left| \int_{I} \int_{\mathbb{R}^3} P_L\left( \int_{-\infty}^{t-|x|} w \dtau \right) (|\nabla| \ftilN) \prod_{j=1}^5 P_{K_j} v ~ \dx \dt \right|\\
&\lesssim \| P_L\big( \int_{-\infty}^{t-|x|} w \dtau \big) \|_{L^{\infty}_{t,x}(\sI)}~ \| |x|^{\frac{3}{8}} |\nabla| \ftilN \|_{\ltx{\frac{8}{3}}{\infty}(\sI)}^{\frac{2}{3}} ~ \| |x|^{\frac{1}{3}} (|\nabla|\ftilN)^{\frac{1}{3}} P_{K_5}v\|_{L_{t,x}^6(\sI)} \\
  &\quad \cdot \prod_{j=2}^4 \| |x|^{-\frac{1}{6}} P_{K_j} v \|_{L_{t,x}^6(\sI)} \| |x|^{-\frac{1}{6}} P_{K_1} v \|_{L_{t,x}^6(\sI)}^{\frac{1}{2}} \| P_{K_1} v \|_{L_t^\infty L_x^2 (\sI)}^{\frac{1}{2}} 
\end{align*}
The first factor is controlled by  \begin{equation*}
 \| P_L\big( \int_{-\infty}^{t-|x|} w \dtau \big) \|_{L^{\infty}_{t,x}(\sI)} \lesssim  \| \int_{-\infty}^{t-|x|} w \dtau \|_{L^{\infty}_{t,x}(\mathbb{R}\times \rthree} \leq \| w \|_{L_\tau^1(\mathbb{R})}~.
\end{equation*}
Arguing as in the proof of Proposition \ref{major:prop_energy_increment}, this leads to the total contribution
\begin{equation*}
\lesssim \| w \|_{L_\tau^1} \| F \|_{Y_I}^{\frac{2}{3}} ( \mathcal{F}_{I} + \| F\|_Z^2 \mathcal{A}_I )^{\frac{1}{6}} \mathcal{A}_I^\frac{7}{12} \mathcal{E}_I^{\frac{1}{4}}  ~.
\end{equation*}
\textbf{Case 2: \( L\geq K_1 \).} In this case, the most severe term is the low-frequency scenario \( {K_1 = \hdots K_5 = 1} \). Then, we can no longer place \( P_{K_1} v \) in \( L_t^\infty L_x^2(\sI) \) and therefore lack  space-integrability. To resolve this, we make use of the integrability of \( w(t-|x|) \) in time. \\
\textbf{Subcase 2.(a): \( L \geq K_1, |x|\geq  1\).} Using Proposition \ref{prelim:prop_mikhlin}, Corollary \ref{major:cor_freq_flux} and Lemma \ref{lem_radial_sobolev}, we obtain that 
\begin{align*}
&\left| \int_{I} \int_{|x|\geq 1} P_L\left( \int_{-\infty}^{t-|x|} w \dtau \right) (|\nabla| \ftilN) \prod_{j=1}^5 P_{K_j} v ~ \dx \dt \right| \\
&\leq   \| \langle x \rangle^{-2} P_L \left( \int_{-\infty}^{t-|x|} w \dtau \right) \|_{L_{t,x}^2(\sI)} \prod_{j=3}^5 \left( \| |x|^{\frac{1}{3}} (|\nabla|\ftilN)^{\frac{1}{3}} P_{K_j} v  \|_{L_{t,x}^6(\sI)}\right) \\
&~~\cdot \prod_{j=1}^2  \| |x|^{\frac{1}{2}} P_{K_j} v \|_{L_{t,x}^\infty(\sI)} \\
&\lesssim N^{1-\delta} \| \langle x \rangle^{-2} P_L \left( \int_{-\infty}^{t-|x|} w \dtau \right) \|_{L_{t,x}^2(\sI)} (\mathcal{F}_I + \| F \|_Z^2 \mathcal{A}_I )^\frac{1}{2} \mathcal{E}_I^{\frac{1}{2}}
\end{align*}
It remains to control the weighted \( L_{t,x}^2 \)-norm.  We recall that the kernel of \( P_L \) has zero mean. Using Lemma 
\ref{prelim:lem_weighted_sobolev} and the boundedness of the Hardy-Littlewood maximal function \( M \), we obtain that 
\begin{align*}
& \| \langle x \rangle^{-2} P_L \left( \int_{-\infty}^{t-|x|} w \dtau \right) \|_{L_{t,x}^2(\sI)}\\
 &= \| \langle x \rangle^{-2} P_L \left( \int_{t-|x|}^{t} w \dtau \right) \|_{L_{t,x}^2(\sI)}\\
  &\lesssim L^{-1} \| \langle x \rangle^{-2} w(t-|x|) \|_{L_{t,x}^2(\sI)} + L^{-1} \| \langle x \rangle^{-3} \int_{t-|x|}^{t} w(\tau) \dtau \|_{L_{t,x}^2(\sI)} \\
  &\lesssim L^{-1} \| \langle x \rangle^{-2} w(t-|x|) \|_{L_{t,x}^2(\sI)} + L^{-1} \| \langle x \rangle^{-3} |x| (Mw)(t-|x|) \|_{L_{t,x}^2(\sI)} \\
  &\lesssim L^{-1} \| \langle x \rangle^{-2} \|_{L_x^2(\rthree)} \left( \| w(t) \|_{L_t^2(\mathbb{R})} + \| Mw(t)\|_{L_t^2(\mathbb{R})} \right) \\
  &\lesssim L^{-1} \| w \|_{L_t^2(\mathbb{R})}~. 
\end{align*}
Putting everything together, it follows that
\begin{align*}
&\left| \int_{I} \int_{|x|\geq 1} P_L \left( \int_{-\infty}^{t-|x|} w \dtau \right) (|\nabla| \ftilN) \prod_{j=1}^5 P_{K_j} v ~ \dx \dt \right| \\
&\lesssim \left( \frac{N}{L} \right)^{1-\delta} L^{-\delta}  \| w \|_{L_\tau^2(\mathbb{R})}
 \left( \mathcal{F}_{I} + \| F \|_Z^2 \mathcal{A}_I  \right)^{\frac{1}{2}} 
\mathcal{E}_I^{\frac{1}{2}}~.
\end{align*}
Using the decay \( L^{-\delta} \) in the highest frequency, we may sum \( N, L, K_1,\hdots K_5 \). \\
\textbf{Subcase 2.(b): \( L \geq K_1, |x|\leq 1\).} Near the origin, our strongest tool  is the Morawetz estimate. Thus, we write 
\begin{align*}
&\left| \int_{I} \int_{|x|\leq 1} P_L \left( \int_{-\infty}^{t-|x|} w \dtau \right) (|\nabla| \ftilN) \prod_{j=1}^5 P_{K_j} v ~ \dx \dt \right|\\
&\lesssim  \| |x|^{\frac{1}{6}} P_L\Big( \int_{-\infty}^{t-|x|} w \dtau \Big) \|_{L_{t,x}^{12}(I\times \{ |x|\leq 1 \})} \| |x|^\frac{2}{3} |\nabla| \ftilN \|_{L_{t,x}^{12}(\sI)}  \prod_{j=1}^5 \| |x|^{-\frac{1}{6}} P_{K_j} v \|_{L_{t,x}^6(\sI)} \\
&\lesssim N^{1-\delta} \| \langle x \rangle^{-1} P_L\Big( \int_{-\infty}^{t-|x|} w \dtau \Big) \|_{L_{t,x}^{12}(\sI)}  \| \ftil \|_{Y_I}  ~ \mathcal{A}_I^{\frac{5}{6}} ~.
\end{align*}
Using Lemma \ref{prelim:lem_weighted_sobolev}, we have that 
\begin{align*}
&\| \langle x \rangle^{-1} P_L\Big( \int_{-\infty}^{t-|x|} w \dtau \Big) \|_{L_{t,x}^{12}(\sI)}  \\
&\leq \| \langle x \rangle^{-1} P_L\Big( \int_{t-|x|}^{t} w \dtau \Big) \|_{L_{t,x}^{12}(\mathbb{R}\times \mathbb{R}^3)} \\
&\lesssim L^{-1} \| \langle x \rangle^{-1} w(t-|x|) \|_{L_{t,x}^{12}(\mathbb{R}\times \mathbb{R}^3)} + L^{-1} \| \langle x \rangle^{-2} \Big( \int_{t-|x|}^{t} w \dtau \Big) \|_{L_{t,x}^{12}(\mathbb{R}\times \mathbb{R}^3)} \\
&\lesssim L^{-1} \| \langle x \rangle^{-1} w(t-|x|) \|_{L_{t,x}^{12}(\mathbb{R}\times \mathbb{R}^3)} +  L^{-1} \| \langle x \rangle^{-1} (Mw)(t-|x|) \|_{L_{t,x}^{12}(\mathbb{R}\times \mathbb{R}^3)} \\
&= L^{-1} \| \langle x\rangle^{-1} \|_{L_x^{12}(\rthree)} \left( \| w \|_{L_\tau^{12}(\mathbb{R})} + \| Mw \|_{L_\tau^{12}(\mathbb{R})} \right) \\
&\lesssim L^{-1} \| w \|_{L_\tau^{12}(\mathbb{R})}~. 
\end{align*}
Putting everything together, it follows that
\begin{equation*}
\left| \int_{I} \int_{|x|\leq 1} P_L \left( \int_{-\infty}^{t-|x|} w \dtau \right) (|\nabla| \ftilN) \prod_{j=1}^5 P_{K_j} v ~ \dx \dt \right|
\lesssim \left(\frac{N}{L}\right)^{1-\delta} L^{-\delta} \| w \|_{L_\tau^{12}(\mathbb{R})} \| \ftil \|_{Y_I}  ~ \mathcal{A}_I^{\frac{5}{6}} ~. 
\end{equation*}
Using the decay \( L^{-\delta} \) in the highest frequency, we may sum \( N, L, K_1,\hdots K_5 \). 
\end{proof}

\begin{prop}[Second main error term in interaction flux estimate]\label{major:prop_interaction_flux_2}
Let \( w \in L^1_\tau(\mathbb{R}) \cap L_\tau^{12}(\mathbb{R}) \) be a nonnegative weight.  Let \( F \) be as in Definition \ref{bootstrap:def_auxiliary_norms} and let \( v\colon I \times \rthree \rightarrow \mathbb{R} \) be a solution of \eqref{flux:eq_forced_nlw}. Then, it holds that
\begin{equation*}
\left| \int_I \int_{\mathbb{R}^3} w(t-|x|) F v^5 \dx \dt \right| \lesssim \| w \|_{L_\tau^2(\mathbb{R})} \mathcal{F}_I^{\frac{1}{2}} \mathcal{E}_I^{\frac{1}{2}} + \| w \|_{L_\tau^{12}(\mathbb{R})}\| F \|_{Y_I}  \mathcal{A}_I^{\frac{5}{6}} ~.
\end{equation*}
\end{prop}
\begin{proof}
We follow an easier version of the arguments in the proof of Proposition \ref{major:prop_interaction_flux}. As before, we distinguish the two cases \( |x|\geq 1 \) and \( |x|\leq 1\). First, we have that
\begin{align*}
&\left| \int_I \int_{|x|\geq 1} w(t-|x|) F v^5 \dx \dt \right|\\
&\leq \| |x|^{-2} w(t-|x|) \|_{L_{t,x}^2(|x|\geq 1)} \left( \| \wout[F]^{\frac{1}{3}} v \|_{L_{t,x}^6(\sI)} +\| \win[F]^{\frac{1}{3}} v \|_{L_{t,x}^6(\sI)} 		\right)^3 \| |x|^{\frac{1}{2}} v\|_{L_{t,x}^\infty}^2 \\
&\lesssim \| w \|_{L_\tau^2}  \mathcal{F}_{I}^{\frac{1}{2}} \mathcal{E}_I^{\frac{1}{2}}~. 
\end{align*}
Second, we have that
\begin{align*}
\left| \int_I \int_{|x|\leq 1} w(t-|x|) F v^5 \dx \dt \right|
&\leq \| |x|^{\frac{1}{6}} w(t-|x|) \|_{L_{t,x}^{12}(|x|\leq 1)} \| |x|^{-\frac{1}{6}} v \|_{L_{t,x}^6(\sI)}^5 \| |x|^{\frac{2}{3}} F \|_{L_{t,x}^{12}(\sI)} \\
&\lesssim \| w \|_{L_\tau^{12}} \| F \|_{Y_I} \mathcal{A}_I^{\frac{5}{6}}~.
\end{align*}
\end{proof}

\subsection{Lower order error terms}\label{section:minor}
\begin{lem}[Control of lower order error terms]\label{a_priori:lem_minor_terms}
 Let \( F \) be as in Definition \ref{bootstrap:def_auxiliary_norms} and let \( v \) be a solution of \eqref{flux:eq_forced_nlw}. Then, it holds that
\begin{alignat*}{3}
&\int_{I} \int_{\rthree} |F|^5 \left( |\partial_t v| + \frac{|v|}{|x|} + |\nabla v | \right) \dx \dt &&\lesssim \| F \|_{Y_I}^5 \mathcal{E}_I^{\frac{1}{2}} ~,\\
&\int_{I} \int_{\rthree} |F|^2 |v|^3 \left( |\partial_t v| + \frac{|v|}{|x|} + |\nabla v | \right) \dx \dt &&\lesssim \| F \|_{Y_I}^2 \mathcal{A}_I^{\frac{1}{2}} \mathcal{E}_I^{\frac{1}{2}} ~,\\
&\int_I \int_{\rthree} \frac{1}{|x|} |F| |v|^5 \dx \dt &&\lesssim \| F \|_{Y_I} \mathcal{A}_I^{\frac{5}{6}}~, \\
&\int_I \int_{\rthree} \frac{1}{|x|} |F|^6 \dx\dt &&\lesssim \| F \|_{Y_I}^6~. 
\end{alignat*}
\end{lem}
\begin{proof} Using Hardy's inequality, the first inequality follows from 
\begin{align*}
&\int_{I} \int_{\rthree} |F|^5 \left( |\partial_t v| + \frac{|v|}{|x|} + |\nabla v | \right) \\
&\leq \| F \|_{L_t^5 L_x^{10}(\sI)}^5 ( \| \partial_t v \|_{L_t^\infty L_x^2(\sI)}+ \| \frac{v}{|x|} \|_{L_t^\infty L_x^2(\sI)} + \| \nabla v \|_{L_t^\infty L_x^2(\sI)}    ) \\
&\lesssim \| F \|_{Y_I}^5 \mathcal{E}_I^{\frac{1}{2}} 
\end{align*}
A similar argument yields that 
\begin{align*}
&\int_{I} \int_{\rthree} |F|^2 |v|^3 \left( |\partial_t v| + \frac{|v|}{|x|} + |\nabla v | \right) \dx \dt \\
&\lesssim \| 	|x|^{\frac{1}{4}} F 	\|_{L_t^4L_x^\infty(\sI)}^2 \| |x|^{-\frac{1}{6}} v \|_{L_{t,x}^6(\sI)}^3 \sup_{t\in I} E[v](t)^\frac{1}{2} ~.
\end{align*}
Finally, the third and fourth inequality follow from Hölder's inequality and
\begin{equation*}
 \| |x|^{-\frac{1}{6}} F \|_{L_{t,x}^6(\sI)} \leq \| F \|_{Y_I} ~.
 \end{equation*}
\end{proof}

\section{Proof of the main theorem}\label{section:a_priori_bound}
In this section, we collect all previous estimates to prove the \textit{a priori} energy bound (Theorem \ref{thm:a_priori_bound}). Using the conditional scattering result of \cite{DLM17}, we finish the proof of Theorem \ref{thm:main}.

\begin{proof}[\textbf{Proof of Theorem \ref{thm:a_priori_bound}}] ~\\
By time-reversal symmetry, it suffices to prove that \( \sup_{t\in [0,\infty)} E[v](t) < \infty \). Let \( \frac{1}{2}\geq \eta_0> 0 \) be a sufficiently small absolute constant, and let \( \frac{1}{2} \geq \eta > 0 \) be sufficiently small depending on \( \eta_0 \). In the following, \( C = C(\|F\|_Z)> 0 \) denotes a large positive constant that depends only on \( \| F \|_Z \).  By Lemma \ref{bootstrap:lem_auxiliary_norms} and space-time divisibility, we can choose a finite partition \( I_1,\hdots, I_J \) of \( [0,\infty) \) such that \( \| F \|_{Y_{I_j}} < \eta \) for all \( j=1,\hdots, J. \) With a slight abuse of notation, we write \( \mathcal{E}_j := \mathcal{E}_{I_j}, \mathcal{A}_{j} := \mathcal{A}_{I_j}\) and \( \mathcal{F}_j := \mathcal{F}_{I_j}\). We also set \( \mathcal{E}_0 := E[v](0) \). \\

First, we estimate the energy increment. Combining Proposition \ref{flux:prop_energy_increment}, Proposition \ref{major:prop_energy_increment}, and Lemma \ref{a_priori:lem_minor_terms}, we have that 
\begin{align}
\mathcal{E}_{j+1} &\leq \mathcal{E}_j +  C \fyjp^{\frac{2}{3}} ( \fjp + \ajp \| F \|_Z^2)^{\frac{1}{6}} \ajp^{\frac{7}{12}} \ejp^{\frac{1}{4}} \notag \\
					&~~~+ C \fyjp^2 \ajp^\frac{1}{2} \ejp^{\frac{1}{2}} + C \fyjp^5 \ejp^{\frac{1}{2}} \notag \\
					&\leq C ( \mathcal{E}_j + 1 ) + \eta_0 \ejp + \eta_0 (\fjp + \ajp)~.  \label{a_priori:eq_proof_energy}
\end{align}
Next, we estimate the Morawetz term. By combining Proposition \ref{flux:lem_morawetz_estimate}, Proposition 
\ref{major:prop_morawetz},  and Lemma \ref{a_priori:lem_minor_terms}, we have that 
\begin{align}
\ajp &\leq C \ejp + C \fyjp^{\frac{2}{3}} \left( \fjp + \ajp \| F \|_Z^2 \right)^{\frac{1}{6}} \ajp^{\frac{7}{12}+\frac{\delta}{6}} \ejp^{\frac{1}{4}-\frac{\delta}{2}} \notag \\
	&~~~+ C \fyjp \ajp^{\frac{5}{6}} + C \fyjp^6 \notag \\
	&\leq C ( \ejp +1 ) + \tfrac{1}{4}( \fjp + \ajp) ~.  \label{a_priori:eq_proof_morawetz}
\end{align}

Finally, we control the interaction flux term. First, recall that from the definition of \( \| F \|_Z \) and the embedding \( \ell_1 \hookrightarrow \ell_2 \), we have that 
\begin{align*}
&\sum_{*\in \{ \text{out},\text{in} \} } \sum_{p\in \{ 2,4,24 \}} \left( \sum_{N\geq 2^5 } (N^{-\frac{1}{6\gamma}+2\delta} + N^{-2+2\delta}) \left( \| W_{*}[|\nabla|F_N] \|_{L_\tau^p}^2 + \| W_{*,\nabla}[F_N]\|_{L_\tau^p}^2  \right)
+ \| W_*[F] \|_{L_\tau^p}^2 \right) \\
&\lesssim \| F \|_Z^2~. 
\end{align*}
We now apply our estimates to each of the terms in \eqref{bootstrap:eq_flux_tilde}, \eqref{bootstrap:eq_flux_grad}, and \eqref{bootstrap:eq_flux_F} separately. By Young's inequality, the estimate \( \| S_K w \|_{L_\tau^p} \lesssim_p \| w \|_{L_\tau^p } \) holds uniformly in \( K \). Using the control on the main and lower order error terms, i.e., 
Proposition \ref{flux:prop_forward_interaction_flux}, Proposition \ref{flux:prop_backward_interaction_flux},
Proposition \ref{major:prop_interaction_flux}, Proposition \ref{major:prop_interaction_flux_2} and Lemma \ref{a_priori:lem_minor_terms}, we obtain that 
\begin{align}
\mathcal{F}_{j+1} &\leq C \| F \|_Z^2 \ejp + C \| F \|_Z^2 \fyjp^{\frac{2}{3}} \left( \fjp + \| F \|_Z^2 \ajp \right)^{\frac{1}{6}} \ajp^{\frac{7}{12}} \ejp^{\frac{1}{4}} \notag \\
				&~~~+ C \| F \|_Z^2\left( \fjp + \| F \|_Z^2 \ajp \right)^{\frac{1}{2}} \ejp^{\frac{1}{2}} \notag\\
				&~~~+ C \| F \|_Z^2 \fyjp \ajp^{\frac{5}{6}} + C \| F \|_Z^2 \fjp^{\frac{1}{2}} \ejp^{\frac{1}{2}}\notag \\
				&~~~+ C \| F \|_Z^2 \fyjp^2 \left( \fyjp^3 + \ajp^{\frac{1}{2}} \right) \ejp^\frac{1}{2}\notag \\
				&\leq C (\ejp + 1) + \tfrac{1}{4} (\fjp + \ajp)~. \label{a_priori:eq_proof_flux}
\end{align}
We briefly note that, as long as \( C> 0 \) remains independent of \( \eta_0 \) and \(\eta\),  terms such as \( C \| F \|_Z^2\fjp^{\frac{1}{2}} \ejp^{\frac{1}{2}} \) prevent us from placing an \( \eta_0 \) in front of \( \fjp + \ajp \). Combining \eqref{a_priori:eq_proof_energy},  \eqref{a_priori:eq_proof_morawetz},  and \eqref{a_priori:eq_proof_flux}, we arrive at 
\begin{align*}
\ejp &\leq C(\mathcal{E}_j +1 ) + \eta_0 \ejp + \eta_0 (\ajp+\fjp)~, \\
\ajp + \fjp &\leq C( \ejp +1 ) + \tfrac{1}{2} (\ajp + \fjp)~.
\end{align*}
Finally, choosing \( \eta_0 > 0 \) sufficiently small depending on \( C= C(\| F \|_Z ) \), we obtain that 
\begin{equation}
\ejp +1 \leq \tilde{C}~ ( \mathcal{E}_j+1) ~. 
\end{equation}
By iterating this inequality finitely many times, we obtain that
\begin{equation}
\sup_{t\in[0,\infty)} E[v](t) = \max_{j=1,\hdots,J} \mathcal{E}_j < \infty~.
\end{equation}
\end{proof}

\begin{proof}[\textbf{Proof of Theorem \ref{thm:main}}]
Using Lemma \ref{lwp:lem_lwp} and Lemma \ref{bootstrap:lem_auxiliary_norms}, it follows that the forced nonlinear wave equation \eqref{intro:eq_forced_nlw} is almost surely locally well-posed. Then, Theorem \ref{thm:main} follows from Theorem \ref{thm:a_priori_bound} and Proposition 
\ref{lwp:prop_scattering}. 
\end{proof}

\bibliography{Library_Wiki}
\bibliographystyle{hplain}

\Addresses

\end{document}